\documentclass[12pt]{article}

\title{Essential tori in spaces of symplectic embeddings}
\author{Julian Chaidez, Mihai Munteanu}

\usepackage{eucal}
\usepackage{amssymb}
\usepackage{latexsym}
\usepackage{tikz-cd}
\usepackage{amsmath}
\usepackage{amsthm}
\usepackage{amscd}
\usepackage{overpic}

\usepackage{url}

\usepackage{color}

\numberwithin{equation}{section}

\newtheorem{theorem}{Theorem}[section]
\newtheorem{proposition}[theorem]{Proposition}

\newtheorem{lemma}[theorem]{Lemma}
\newtheorem{lemma-definition}[theorem]{Lemma-Definition}

\theoremstyle{definition}

\newtheorem{definition}[theorem]{Definition}
\newtheorem{remark}[theorem]{Remark}

\newtheorem{review}[theorem]{Review}
\newtheorem{example}[theorem]{Example}

\newtheorem{notation}[theorem]{Notation}

\newcommand{\C}{{\mathbb C}}
\newcommand{\Q}{{\mathbb Q}}
\newcommand{\R}{{\mathbb R}}

\newcommand{\Z}{{\mathbb Z}}

\newcommand{\op}{\operatorname}

\newcommand{\CZ}{\op{CZ}}

\newcommand{\bpm}{\begin{pmatrix}}
\newcommand{\epm}{\end{pmatrix}}

\begin{document}

\setcounter{tocdepth}{2}

\maketitle

\begin{abstract}
Given two $2n$--dimensional symplectic ellipsoids whose symplectic sizes satisfy certain inequalities, we show that a certain map from the $n$--torus to the space of symplectic embeddings from one ellipsoid to the other induces an injective map on singular homology with mod $2$ coefficients. The proof uses parametrized moduli spaces of $J$--holomorphic cylinders in completed symplectic cobordisms.
\end{abstract}

\section{Introduction}

The study of symplectic embeddings is a major area of focus in symplectic geometry. Remarkably, the space of such embeddings can have a rich and complex structure, even when the domain and target manifolds are relatively simple.

Symplectic embeddings between ellipsoids are a well--studied instance of this phenomenon. For a nondecreasing sequence of positive real numbers $a=(a_1,a_2,\dots,a_n)$ define the \textit{symplectic ellipsoid} $E(a)$ by 
\begin{equation}\label{equation:ellipsoid}
E(a)=E(a_1,a_2,\dots,a_n) := \left \lbrace (z_1,\dots, z_n) \in  \C ^n\; \middle| \; \sum_{i=1}^n \frac{\pi |z_i|^2}{a_i} \leq 1 \right \rbrace.
\end{equation}
The space $E(a)$ carries the structure of an exact symplectic manifold with boundary endowed with the restriction of the standard Liouville form $\lambda$ on $\C^n$, given by
\begin{equation}\label{equation:standard_liouville_form}
\lambda = \frac{1}{2}\sum_{i=1}^n (x_idy_i - y_idx_i).\end{equation}
A special case is the \emph{symplectic ball} $B^{2n}(r)$, which is simply $E(a)$ for $a = (r,\dots,r)$.

The types of results that one can prove about symplectic embeddings, together with the tools used to do so, are surveyed at length by Schlenk in \cite{schlenk2018}. Most research has thus far  sought to address the existence problem. Let us recall some of the more striking progress in this direction. The first nontrivial result was Gromov's eponymous \emph{nonsqueezing theorem}, proven in the seminal paper \cite{gromov1985}.
\begin{theorem}[{\cite{gromov1985}}] 
\label{theorem:gromov_nonsqueezing} 
There exists a symplectic embedding 
\[ B^{2n}(r) \to B^2(R)\times \C^{2n-2}\]
 if and only if $r\leq R$.
\end{theorem}
This result demonstrated that there are obstructions to symplectic embeddings beyond the volume and initiated the study of quantitative symplectic geometry. Note that Theorem \ref{theorem:gromov_nonsqueezing} can be seen as a result about ellipsoid embeddings, since $B^{2}(R) \times \mathbb{C}^{2n-2}$ can be viewed as the degenerate ellipsoid $E(R,\infty, \dots, \infty)$. 

In dimension $4$, the question of when the ellipsoid $E(a,b)$ symplectically embeds into the ellipsoid $E(a',b')$ was answered by McDuff in \cite{mcduff2011}. Let $\{N_k(a,b)\}_{k\geq 0}$ denote the sequence of nonnegative integer linear combinations of $a$ and $b$, ordered nondecreasingly with repetitions.
\begin{theorem}[\cite{mcduff2011}]
\label{theorem:mcduff}
There exists a symplectic embedding 
\[\op{int} (E(a,b)) \to E(a',b')\]
 if and only if $N_k(a,b)\leq N_k(a',b')$ for every nonnegative integer $k$. 
 \end{theorem}
A special case of this embedding problem, where the target ellipsoid is the ball $B^4(\lambda)$, was studied by McDuff and Schlenk in an earlier paper \cite{mcduff_schlenk2012} using methods different from \cite{mcduff2011}. In that paper, McDuff and Schlenk give a remarkable calculation of the function $c_0:\R^+ \to \R^+$ defined by
\[ c_0(a):=\inf \left\lbrace \lambda \; \middle| \;  E(1,a) \; \text{symplectically embeds into} \; B^4(\lambda) \right\rbrace. \] 
In particular, they show that for $a \in \left[1,\right(\frac{1+\sqrt{5}}{2}\left)^4\right]$, the function $c_0$ is given by a piecewise linear function involving the Fibonacci numbers, which they call the \emph{Fibonacci staircase}. Some higher dimensional cases of the existence problem for symplectic embeddings have been studied in a similar manner. For instance, a family of stabilized analogues of the function $c_0$, which are defined as
\[ c_n(a):=\inf \left\lbrace \lambda \; \middle| \;  E(1,a) \times \C^n\; \text{symplectically embeds into} \; B^4(\lambda) \times \C^n \right\rbrace, \] 
are studied in the more recent papers \cite{cristofaro_hind2015} and \cite{cristofaro_hind_mcduff2017}. 

Beyond problems of existence, one can ask about the algebraic topology of the space of symplectic embeddings $\op{SympEmb}(U,V)$ between two symplectic manifolds $U$ and $V$, with respect to the $C^\infty$ topology. Again, most results have been proven in dimensions 2 and 4. For instance, in \cite{mcduff2009}, McDuff demonstrated that the space of embeddings between $4$--dimensional symplectic ellipsoids is connected whenever it is nonempty. Other results in dimension $4$ can be found in \cite{anjos2009} and \cite{hind_pinsonnault2013}. 

More recently, in \cite{munteanu2018}, the second author developed methods to show that the contractibility of certain loops of symplectic embeddings of ellipsoids depends on the relative sizes of the two ellipsoids.  

\subsection{Main result}

In this paper, we build upon the methods developed in \cite{munteanu2018} to tackle the question of describing the higher homology groups of spaces of symplectic embeddings between ellipsoids in any dimension. 

More precisely, we will be studying families of symplectic embeddings that are restrictions of the following unitary maps. For $\theta =(\theta_1,\dots,\theta_n)\in T^n = (\R/2\pi \Z)^n$, let $U_{\theta}$ denote the unitary transformation
\begin{equation}
\label{equation:torus}
 U_{{\theta}}(z_1, \dots, z_n) :=(e^{i\theta_1} z_1,  \dots, e^{i\theta_n} z_n).
\end{equation}
Given symplectic ellipsoids $E(a)$ and $E(b)$ such that $a_i < b_i$ for every $i\in\{1 \dots, n\}$, we may define the family of ellipsoid embeddings
\begin{equation}
\label{equation:embedding_family}
\Phi:T^n \to \op{SympEmb}(E(a),E(b)), \qquad \Phi(\theta) = U_\theta|_{E(a)}
\end{equation}
by restricting the domain of the maps $U_{\theta}$. The following theorem about the family $\Phi$ is the main result of this paper.

\begin{theorem}[Main theorem]
\label{theorem:maintheorem}
Let $a=(a_1,\dots,a_n)$ and $b=(b_1,\dots, b_n)$ be two sequences of real numbers satisfying
\[a_i < b_i < a_{i+1} \quad \text{for all} \quad i\in \{1,\dots, n-1\} \quad \text{and} \quad a_n < b_n < 2a_1.\]
Furthermore, let $\Phi:T^n \to \op{SympEmb}(E(a),E(b))$ be the family of symplectic embeddings (\ref{equation:embedding_family}). Then the induced map
\[ \Phi_{*}: H_*(T^n;\Z/2) \rightarrow H_*(\op{SympEmb}(E(a), E(b));\Z/2)\]
 on homology with $\Z/2$--coefficients is injective.
\end{theorem}

In order to demonstrate the nontriviality of Theorem \ref{theorem:maintheorem}, we note that the map induced by $\Phi:T^n \to \op{Symp}(E(a),E(b))$ on $\Z/2$--homology has a sizeable kernel when $E(a)$ is very small relative to $E(b)$. More precisely, we have the following.

\begin{proposition}
\label{prop:smallellipsoid}
Let $a=(a_1,\dots,a_n)$ and $b=(b_1,\dots, b_n)$  be two nondecreasing sequences of real numbers satisfying $a_n < b_1$. Furthermore, let $\Phi:T^n \to \op{SympEmb}(E(a),E(b))$ be as in (\ref{equation:embedding_family}). Then the induced map $\Phi_*$ on $\Z/2$--homology has rank $1$ in degree $\le 1$ and rank $0$ otherwise.
\end{proposition}

Unlike the proof of Theorem \ref{theorem:maintheorem}, the proof of Proposition \ref{prop:smallellipsoid} is an elementary calculation in algebraic topology which we defer to \S \ref{section:proof}. 

\begin{remark}[Comparison to \cite{munteanu2018}]
 \label{remark:comparison_to_mihai} 
 In dimension 4, the fact that $\Phi_*$ is injective in degree $1$ was proven by the second author in \cite{munteanu2018}. Specifically, this is equivalent to \cite[Theorem 1.4]{munteanu2018} which states that the loop
\[ \Psi:S^1 \to \op{Symp}(E(a),E(b)) \]
defined by 
\[ \Psi(t)(z_1,z_2):= 
\left\{\begin{array}{cc}
(e^{4\pi i t}z_1, z_2) & t \in \left[0,\frac{1}{2}\right]\\
(e^{-4\pi i t}z_1, z_2) & t \in \left(\frac{1}{2},1\right]\\
\end{array}\right.
\]
is noncontractible. In fact, \cite{munteanu2018} actually addresses the more general 4--dimensional case where $E(a)$ and $E(b)$ are replaced with convex toric domains in $\C^2$ satisfying specific inequalities involving their ECH capacities. We expect Theorem \ref{theorem:maintheorem} to hold at this level of generality, and we hope to address this in future work using somewhat different methods (see Remark \ref{remark:lagrangian_generalization}).
\end{remark}

\begin{remark}[$\Z$ vs $\Z/2$ coefficients]
\label{remark:ZvsZ2_coefficients} 
Our use of $\Z/2$ coefficients, instead of $\Z$ coefficients, allows us to use the methods of \S \ref{section:topology} to work entirely with smooth manifolds with boundary as opposed to cochains. While the contents of \S \ref{section:topology} provide a nice technical work around, we expect Theorem \ref{theorem:maintheorem} to hold at the level of $\Z$ coefficients as well. We plan to develop the methods needed to work over $\Z$ in forthcoming work.
\end{remark}

\begin{remark}[Lagrangian analogues]
\label{remark:lagrangian_generalization}
In forthcoming work, we hope to demonstrate results analogous to Theorem \ref{theorem:maintheorem} for families of Lagrangian torus embeddings in toric domains. We anticipate that these results will be useful for demonstrating the various generalization of Theorem \ref{theorem:maintheorem} discussed in Remark \ref{remark:comparison_to_mihai}.
\end{remark}

\noindent \textbf{Organization.} The rest of the paper is organized as follows. In \S\ref{section:setup}, we establish the geometric setup and notation. In \S \ref{section:proof}, we construct the moduli spaces and prove the needed transversality and compactness properties together with a lemma about the count of curves in these moduli spaces to build up towards a proof of Theorem \ref{theorem:maintheorem}. Lastly, in \S \ref{section:topology}, we prove some useful technical results about the topology of spaces of symplectic embeddings.

\vspace{10pt}

\noindent \textbf{Acknowledgements.} We would like to thank our advisor, Michael Hutchings for all the helpful discussions and for pointing our some significant simplifications to earlier drafts. JC was supported by the NSF Graduate Research Fellowship under Grant No. 1752814. MM was partially supported by NSF Grant No. DMS--1708899.

\section{Geometric setup}
\label{section:setup}

In this section, we review the concepts from contact geometry and holomorphic curve theory needed in this paper. For a more comprehensive discussion of these topics, see \cite{geiges2008}, \cite{mcduff2017}, \cite{wendl2010} and \cite{wendl2016}.

\subsection{Contact geometry} \label{subsection:review_of_contact_geometry} We begin by providing a quick overview of basic contact geometry and establishing notation for \S \ref{section:proof}. We include a review the Reeb dynamics on the boundary of a symplectic ellipsoid with rationally independent defining parameters.

\begin{review}[Contact manifolds] Recall that a \emph{contact manifold} $(Y,\xi)$ is a smooth $(2n-1)$--manifold $Y$ together with a rank $2n-2$ sub-bundle $\xi \subset TY$ that is given fiberwise by the kernel $\xi = \op{ker}(\alpha)$ of a contact $1$--form $\alpha \in \Omega^1(Y)$. A \emph{contact form} $\alpha$ is a $1$--form on $Y$ satisfying $\alpha \wedge d\alpha^{n-1} \neq 0$ everywhere. 

Every contact form $\alpha$ on $Y$ has a naturally associated \emph{Reeb vector field} $R_\alpha$ defined implicitly from $\alpha$ via the equations
\begin{equation}
\label{equation:Reeb_vectorfield}
\iota_{R_\alpha}\alpha = 1, \qquad \iota_{R_\alpha}d\alpha = 0.
\end{equation}
The \emph{Reeb flow} $\Phi_\alpha:Y \times \R \to Y$ is the flow of the vector field $R_\alpha$, i.e. the family of diffeomorphisms satisfying
\begin{equation}
\label{equation:Reeb_flow}
\frac{d\Phi^t_\alpha(y)}{dt}\bigg\rvert_{t = s} = R_\alpha \circ \Phi^s_\alpha(y).
\end{equation}

A \emph{Reeb orbit} is a closed orbit of the flow $\Phi_\alpha$, i.e. a curve $\gamma:S^1=\R/L\Z \to Y$ satisfying $\frac{d \gamma}{d t} = R_{\alpha}\circ \gamma$ for some positive number $L$ which is called the \textit{period}. Note that $L$ coincides with the \emph{action} $\mathcal{A}_\alpha(\gamma)$ of $\gamma$, which is defined as
\begin{equation}
\mathcal{A}_\alpha(\gamma) := \int_{S^1} \gamma^*\alpha.
\end{equation}

A Reeb orbit $\gamma$ is called \emph{nondegenerate} if the differential $T\Phi^L_{\gamma(0)}$ of the time $L$ flow satisfies
\begin{equation} \label{equation:nondegenerate}
\op{det}(T\Phi^L_{\gamma(0)}|_\xi - \op{Id}_{\xi}) \neq 0.
\end{equation}
A contact form $\alpha$ is called \textit{nondegenerate} if every Reeb orbit of $\alpha$ is nondegenerate. \end{review}

\begin{review}[Conley--Zehnder indices] \label{review:conley_zehnder_indices} Any nondegenerate Reeb orbit $\gamma$ posseses a fundamental numerical invariant called the \emph{Conley--Zehnder index} $\CZ(\gamma,\tau)$, whose definition and computation we now review. 

 The Conley--Zehnder index $\CZ(\gamma,\tau)$ depends on a choice of symplectic trivialization $\tau:\gamma^*\xi \simeq S^1 \times \C^{n-1}$. The invariant is defined by $\CZ(\gamma,\tau) := \mu_{\op{RS}}(\phi)$ were $\mu_{RS}$ denotes the Robbin--Salamon index defined in \cite{rs1993} and $\phi$ is the path of symplectic matrices defined as
 \[
  \phi:[0,L] \to \op{Sp}(2n-2), \qquad \phi(t) :=\tau_{\gamma(t)} \circ T\Phi^{t}_{\gamma(0)}|_{\xi} \circ \tau_{\gamma(0)}^{-1}.
 \]

In the case where $c_1(\xi) = 0 \in H^2(Y;\Z)$ and $[\gamma] = 0 \in H_1(Y;\Z)$, a canonical Conley--Zehnder index $\op{CZ}(\gamma)$ which does not depend on a choice of trivialization can be associated to $\gamma$ via the following procedure. Extend $\gamma$ to a map $u:\Sigma \to Y$ from an oriented surface $\Sigma$ with boundary $\partial \Sigma = S^1$ satisfying $u|_{\partial \Sigma} = \gamma$. Pick a symplectic trivialization $\sigma:u^*\xi \simeq \Sigma \times \C^{n-1}$ and define $\op{CZ}(\gamma)$ by the formula
\begin{equation}
\op{CZ}(\gamma) := \op{CZ}(\gamma,\sigma|_{\partial \Sigma}).
\end{equation}
The fact that $\op{CZ}(\gamma)$ is independent of $\Sigma$ and $\sigma$ follows from the vanishing of the first Chern class. The index $\op{CZ}(\gamma)$ can be related to the index $\op{CZ}(\gamma,\tau)$ with respect to a trivialization $\tau$ by the formula
\begin{equation}\CZ(\gamma) = \CZ(\gamma,\tau) + 2c_1(\gamma,\tau).\end{equation}
Here $c_1(\gamma,\tau)$ is the relative first Chern number with respect to $\tau$ of the pullback $u^*\xi$ of $\xi$ to a capping surface $u$ of $\gamma$.\end{review}

For the purposes of this paper, we are interested in a specific family of examples of contact manifolds, namely boundaries $\partial E(a)$ of irrational symplectic ellipsoids.  

\begin{example}[Ellipsoids]
 \label{example:ellipsoid_dynamics} 
Let $E(a)$ be a symplectic ellipsoid with parameters $a = (a_1,\dots,a_n) \in (0,\infty)^n$ satisfying $a_i/a_j \not\in \Q$ for each $i \neq j$. The boundary of the ellipsoid $\partial E(a)$ together with the restriction of the standard Liouville form $\lambda$ on $\C^n$ defined by (\ref{equation:standard_liouville_form}) is a contact manifold. 

The discussion in the proof of \cite[Lemma 2.1]{gutthutchings2018} shows that there are precisely $n$ simple orbits $\gamma_i=\{(z_1,z_2,\dots,z_n)\in\partial E(a) \; | \; z_j=0, \; \forall j\neq i\}$ for $1 \le i \le n$. All the orbits $\gamma_i$ are nondegenerate and their action is given by $\mathcal{A}_\alpha(\gamma_i) = a_i$. Moreover, using the linearization of the Reeb flow, one can compute the Conley--Zehnder indices of the Reeb orbits $\gamma_i^m$ to be
\begin{equation}
\CZ(\gamma^m_i) =  \sum_{j \neq i} \left(2 \left\lfloor \frac{ma_i}{a_j} \right\rfloor + 1\right) + 2m,
\end{equation}
which after some smart rewriting becomes
\begin{equation} \label{equation:CZ_index_for_ellipsoids}
\CZ(\gamma^m_i) = n - 1 + 2\left|\left\lbrace L \in \op{Spec}(Y,\alpha)\;\middle|\; L \le ma_i\right\rbrace\right|.
\end{equation}
\end{example}

Next, we review the basic terminology of exact symplectic cobordisms and associated structures. Throughout the discussion for the rest of the section, let $(Y_{\pm},\alpha_{\pm})$ be closed contact $(2n-1)$--manifolds with contact forms $\alpha_\pm$.

\begin{review}[Exact symplectic cobordisms] \label{review:exact_symplectic_cobordism} Recall that an \textit{exact symplectic cobordism} $(W,\lambda,\iota)$ from $(Y_{+},\alpha_{+})$ to $(Y_{-},\alpha_{-})$ consists of the following data.
\begin{itemize}
\item[$\cdot$] A compact, exact symplectic manifold $(W,\lambda)$ with boundary $\partial W$ such that the \emph{Liouville vector field} $Z$ defined by the equation $d\lambda(Z,\cdot) = \lambda$ is transverse to $\partial W$ everywhere. In this situation, $\partial W = \partial_+W \sqcup \partial_-W$, where $Z$ points outward along $\partial_+W$ and inward along $\partial_-W$.
\item[$\cdot$] A pair of boundary inclusion maps $\iota_+$ and $\iota_-$, which are strict contactomorphisms of the form
\begin{equation}
\iota_+:(Y_+,\alpha_+) \simeq (\partial_+W,\lambda|_{\partial_+ W}), \qquad \iota_-:(Y_-,\alpha_-) \simeq (\partial_-W,\lambda|_{\partial_- W}).
\end{equation}
\end{itemize}
We will generally suppress the inclusions in the notation, using $\iota_+$ and $\iota_-$ when needed. The maps $\iota_+$ and $\iota_-$ extend, via flow along $Z$ or $-Z$, to collar coordinates
\begin{equation} \label{equation:collar_charts} ([0,\epsilon)\times Y_{-}, e^s\lambda_{-}) \simeq (N_{-}, \lambda|_{N_-}), \qquad ((-\epsilon,0]\times Y_{+}, e^s\lambda_{+}) \simeq (N_{+}, \lambda|_{N_+}).\end{equation}
Here $N_-$ and $N_+$ are collar neighborhoods of $Y_-$ and $Y_+$ respectively, the maps preserve the $1$--forms above and $s$ denotes the coordinate on $[0,\epsilon)$ and $(-\epsilon,0]$.

Given exact symplectic cobordisms $(W,\lambda,\iota)$ from $(Y_0,\alpha_0)$ to $(Y_1,\alpha_1)$ and $(W',\lambda',\iota')$ from $(Y_1,\alpha_1)$ to $(Y_2,\alpha_2)$, we can form the \emph{composition} $(W \# W', \lambda \# \lambda', \iota \# \iota')$ by gluing $W$ and $W'$ via the identification $(\iota'_+)^{-1} \circ \iota_-$ of $\partial_-W$ and $\partial_+W'$. The Liouville forms and inclusions extend in the obvious way to the glued manifold.

Using these identifications (\ref{equation:collar_charts}), we can complete the exact symplectic cobordism $(W,\lambda)$ by adding cylindrical ends $(-\infty,0]\times  Y_{-}$ and $[0,\infty) \times Y_{+}$ to obtain the \textit{completed exact symplectic cobordism} $(\widehat{W},\widehat{\lambda})$, given by
\begin{equation} \widehat{W} =  (-\infty, 0]\times Y_{-} \sqcup_{\iota_-} W \sqcup_{\iota_+} [0,\infty)\times Y_{+} .\end{equation}
The Liouville $1$--forms $\lambda$, $e^s\alpha_-$ and $e^s\alpha_+$ glue together to a Liouville form $\widehat{\lambda}$ on $\widehat{W}$. An important special caase of completed cobordisms is given by the \emph{symplectization} of a contact manifold $(\R \times Y,e^s\alpha)$, which we will denote by $\widehat{Y}$.

Given a manifold $P$ (with or without boundary), a \emph{$P$--parametrized family of exact symplectic cobordisms} $\mathfrak{W}$ from $Y_+$ to $Y_-$ is a fiber bundle $\mathfrak{W} \to P$ over $P$ with fiber $W_p$ at $p \in P$, a $1$--form $\lambda$ on $\mathfrak{W}$ and a bundle map $\iota^\pm:P \times Y^\pm \to W$ such that $(W_p,\lambda|_{W_p},\iota_p)$ is an exact symplectic cobordism for each $p \in P$. 
\end{review}

\begin{review} (Almost complex structures) Recall that a compatible almost complex structure $J_\xi$ on the symplectic vector bundle $\xi$ gives rise to an $\R$--invariant compatible almost complex structure $J$ on the symplectization $\R \times Y$, defined by
\[J(\partial_s) = R_\alpha, \qquad J(R_\alpha) = -\partial_s, \qquad J|_\xi = J_\xi.\]
We denote the set of such translation invariant $J$ on $\R \times Y$ by $\mathcal{J}(Y)$.

An almost complex structure $J$ on a completed exact symplectic cobordism $\widehat{W}$ as above is called \textit{compatible} if it has the following properties.
\begin{enumerate}
\item[$\cdot$] On the ends $[0,\infty) \times Y_+$ and $(-\infty,0] \times Y_-$, $J$ restricts to $\R$--invariant complex structures arising from $J_+ \in \mathcal{J}(Y_+)$ and $J_- \in \mathcal{J}(Y_-)$, respectively.
\item[$\cdot$] The almost complex structure $J$ is compatible with the symplectic form $d\lambda$. 
\end{enumerate}
We let $\mathcal{J}(W)$ denote the set of all such compatible almost complex structures on a given exact symplectic cobordism $W$. More generally, given a $P$--parametrized family of exact symplectic cobordisms $\mathfrak{W}$, we denote by $\mathcal{J}(\mathfrak{W})$ the space of smooth, fiberwise almost complex structures $\mathfrak{J}$ such that $J_p \in \mathcal{J}(W_p)$ for each $p \in P$.

We note that $\mathcal{J}(W)$ is contractible for any $W$ (see for instance \cite[Proposition 4.11]{mcduff2017}). This implies that the space of families $\mathcal{J}(\mathfrak{W})$ is also contractible, and that any family $\mathfrak{J}_{\partial P} \in \mathcal{J}(\mathfrak{W}|_{\partial P})$ over $\partial P$ extends to a family $\mathfrak{J} \in \mathcal{J}(\mathfrak{W})$ over all of $P$.
\end{review}

As with contact manifolds, we are interested in a particular family of examples of exact symplectic cobordisms related to ellipsoid embeddings.

\begin{notation}[Cobordisms of embeddings]
\label{example:cobordism_of_embedding}  Let $E(a)$ and $E(b)$ be irrational ellipsoids. Given a symplectic embedding $\varphi: E(a) \to \op{int}(E(b))$, we denote by $W_{\varphi}$ the exact symplectic cobordism given by 
\begin{equation} \label{equation:cobordism_of_embedding} W_\varphi := E(b) \setminus \op{int} (\varphi(E(a))), \qquad \iota_+ := \op{Id}|_{\partial E(b)}, \qquad \iota_- := \varphi|_{\partial E(a)}.
\end{equation}More generally, let $P$ be a compact manifold with boundary and $\Psi:P \times E(a) \to \op{int}(E(b))$ be a $P$--parametrized family of symplectic embeddings. We then acquire a family of cobordisms $\mathfrak{W}_\Psi$ with fiber $(\mathfrak{W}_{\Psi_p},\lambda_{\Psi_p})$ given by (\ref{equation:cobordism_of_embedding}). 

In this context, we label the simple Reeb orbits of $\partial E(b)$ by $\gamma^{+}_{i}$ and the simple Reeb orbits of $\partial E(a)$ by $\gamma^{-}_{i}$. The simple Reeb orbits of the negative boundary of $W_\varphi$ are, of course, the images $\varphi(\gamma^-_i)$ and will be denoted as such. Furthermore, if the image $\op{Im}(\Psi_p)$ of $\Psi_p$ is independent of $p$ sufficiently close to $\partial P$, then we let $W_{\partial P}$ denote $E(b) \setminus \Psi_p(E(a))$ for any $p \in \partial P$ and we let $\lambda_{\partial P}$ denote the Liouville form. Note that in this case, the cobordisms $(W_{\Psi_p},\lambda_{\Psi_p},\iota_{\Psi_p})$ for $p \in \partial P$ differ only by the boundary inclusion $\iota_{\Psi_p}$. In situations where $\iota_{\Psi_p}$ plays no role, we will often not distinguish between $(W_{\Psi_p},\lambda_{\Psi_p},\iota_{\Psi_p})$ for different $p \in \partial P$.
\end{notation}

\subsection{Holomorphic curves and neck stretching} \label{subsection:holomorphic_curves_and_neck_stretching}

The proof of Theorem \ref{theorem:maintheorem} is centered around the analysis of certain moduli spaces of holomorphic curves. In this section, we give a quick overview of holomorphic curves, SFT compactness, and SFT neck stretching.

\begin{definition}[Holomorphic Curve] Let $(W,\lambda)$ be an exact symplectic cobordism from $(Y_+,\alpha_+)$ to $(Y_-,\alpha_-)$, equipped with an almost complex structure $J \in \mathcal{J}(W)$. Let $(\Sigma,j)$ be a Riemann surface acquired by removing a finite set $P_+$ of positive punctures and a finite set $P_-$ of negative punctures from a closed Riemann surface $\overline{\Sigma}$. Finally, let $\Gamma^* = \{\gamma^*_p \; | \; p \in P_*\}$ be a set of Reeb orbits in $Y_*$ for each $* \in \{+,-\}$.

\noindent A (parametrized) holomorphic curve $u:(\Sigma,j) \to (\widehat{W},J)$ asymptotic to $\Gamma^+$ at $P_+$ and $\Gamma^-$ at $P_-$ is a smooth map such that
\begin{itemize}
\item[$\cdot$] $u$ is $(j,J)$--holomorphic, i.e. $J_{u(p)} \circ du_p = du_p \circ j_p$ for all $p \in \Sigma$ and
\item[$\cdot$] for any $p \in P_*$ for $* \in \{+,-\}$, there exists a holomorphic chart $\varphi:S^1 \times \R^* \simeq \overline{\Sigma}$ with $[u \circ \varphi](S^1 \times \R^*) \subset Y_* \times \R^* \subset \widehat{W}$ and
\[\lim_{r \to * \infty} \varphi(\theta,r) = p, \qquad \lim_{r \to * \infty} [\pi_\R \circ u \circ \varphi](\theta,r) = * \infty, \qquad \lim_{r \to *\infty} [\pi_{Y_*} \circ u \circ \varphi](\cdot,r) = \gamma^*_p.
\]
 
\end{itemize}The left-most limit above is taken in the $C^0$--topology. As an alternative to the last two conditions above, we may assert that the limit of $u \circ \varphi(\cdot,\cdot * R)$ converges in $C^0$ to a parametrization of the \emph{trivial cylinder} $\R \times \gamma^*_p$ as $R \to \infty$.

Two (parametrized) holomorphic curves $u:\Sigma \to \widehat{W}$ and $u':\Sigma' \to \widehat{W}$ are equivalent if there is a biholomorphism $\varphi:\Sigma \to \Sigma'$ with $u = u' \circ \varphi$. An (unparametrized) holomorphic curve is a parametrized holomorphic curve up to this equivalence relation. The curves in this paper will be unparametrized, unless otherwise specified.
 \end{definition}

We now provide the reader with brief, very simplified reviews of SFT compactness and SFT neck stretching. We refer the reader to \cite[\S 10]{behwz2003} for the original proofs and to \cite[\S 9.4]{wendl2010} for a detailed overview.

\begin{review}[SFT Compactness]
\label{review:SFT_compactness} Let $P$ be a compact manifold with boundary, and let $(Y_*,\alpha_*)$ for $* \in \{+,-\}$ be closed, nondegenerate contact manifolds. Let $\mathfrak{W}$ be a $P$--paramaterized family of exact symplectic cobordisms from $Y_+$ to $Y_-$ equipped with a $P$--parametrized family $\mathfrak{J} \in \mathcal{J}(\mathfrak{W})$ such that $J_p|_{[0,\infty)\times Y_{+}}=J_{+}$ and $J_p|_{(-\infty,0]\times Y_{-}}=J_{-}$ for some fixed almost complex structures $J_{\pm} \in \mathcal{J}(Y_\pm)$. Fix a surface $\Sigma$, acquired by taking a closed surface $\overline{\Sigma}$ and removing a finite set of punctures. Finally, consider a sequence $p_i \in P$ and $u^i:\Sigma \to (\widehat{W}_{p_i}, J_{p_i})$ of $J_{p_i}$--holomorphic curves asymptoting to collections of Reeb orbits $\Gamma^+$ (at the positive end of $\widehat{W}_{p_i}$) and $\Gamma^-$ (at the negative end of $\widehat{W}_{p_i}$) independent of $i$.

 The SFT compactness theorem states that, after passing to a subsequence, $p_i \to p \in P$ and $u^i$ converges (in the SFT Gromov topology, see \cite[\S 7.3]{behwz2003}) to a \emph{$J_p$--holomorphic building}, which is a tuple of the form
 \begin{equation}\label{equation:building_notation} v = (u^+_1,\dots,u^+_M,u^{W},u^-_1,\dots,u^-_N).\end{equation}
Here $M,N \in \Z^{\ge 0}$ are integers and the elements of the tuple (called \emph{levels}) are holomorphic maps from punctured surfaces of the form
\[u^*_j:S^*_{j} \to (\R\times Y_*, J_*) \text{ for }* \in \{+,-\} \quad\text{and}\quad u^W:S^W \to (\widehat{W}_p,J_p).\]
The maps $u^*_j$ and the map $u^W$ are considered modulo domain reparametrization, and modulo translation when the target manifold is a symplectization. The surfaces $S_{j}$ can be glued together along the boundary punctures asymptotic to matching Reeb orbits, and this glued surface $\#_j S_{j}$ is homeomorphic to $\Sigma$. 

All of the curves $u^*_j$ and $u^W$ must be asymptotic to a Reeb orbit at each positive and negative puncture. We denote the collections of positive and negative limit Reeb orbits of $u^W$ (with multiplicity) by $\Gamma^+(u^W)$ and $\Gamma^-(u^W)$, respectively, and we adopt similar notation for $u^*_j$. The asymptotics of the $u^*_j$ and $u^W$ must be compatible, in the sense that the negative ends of $u^*_j$ and the positive ends of $u^*_{j+1}$ must agree (and likewise for $u^+_M$ and $u^W$, etc.). Furthermore, we must have $\Gamma^+(u^+_1) = \Gamma^+$ and $\Gamma^-(u^-_N) = \Gamma^-$. Finally, every symplectization level $u^*_j$ must have at least one component that is not a trivial cylinder $\R \times \gamma$. 

Since $(W_p,\lambda_p)$ is an exact symplectic cobordism, one may apply Stoke's theorem to derive the following expression for the energies of the levels of $v$:
\begin{equation} \label{equation:energy_stokes_cob}
\mathcal{E}(u^W) := \int_{S^W} [u^W]^*d\lambda_p =  \sum_{\eta^+ \in \Gamma^+(u^W)} \mathcal{A}(\eta^+) - \sum_{\eta^- \in \Gamma^-(u^W)} \mathcal{A}(\eta^-),
\end{equation}
\begin{equation} \label{equation:energy_stokes_symp}
\mathcal{E}(u^\pm_j) := \int_{S_j} [u^\pm_j]^*d(e^t\alpha_\pm) =  \sum_{\eta^+ \in \Gamma^+(u^\pm_j)} \mathcal{A}(\eta^+) - \sum_{\eta^- \in \Gamma^-(u^\pm_j)} \mathcal{A}(\eta^-).
\end{equation}
The positivity of the energy of any holomorphic curve implies that the right hand sides of (\ref{equation:energy_stokes_cob}) and (\ref{equation:energy_stokes_symp}) are nonnegative. More generally, if we let $\mathcal{A}[\Gamma]$ denote the total action of a collection of Reeb orbits,  then we have the string of inequalities
\begin{equation} \label{equation:action_monotonicity_b}
\begin{split}
\mathcal{A}[\Gamma^-] &= \mathcal{A}[\Gamma(u^-_N)] \le \dots \le \mathcal{A}[\Gamma(u^-_1)] \le \mathcal{A}[\Gamma(u^W)] \le \\
	&\le \mathcal{A}[\Gamma(u^+_M)] \le \dots \le \mathcal{A}[\Gamma(u^+_1)] = \mathcal{A}[\Gamma^+].
\end{split}
\end{equation}

There is some additional data, beyond the holomorphic curves themselves, associated to a holomorphic building. However, we suppress this data since it will play no role in any of our arguments below.
\end{review}

\begin{review}[SFT Neck Stretching] \label{review:SFT_neck_stretching} Let $(Y_*,\alpha_*)$ for $* \in \{0,1,2\}$ be closed, nondegenerate contact manifolds and let $(U,\lambda_U,J_U)$ and $(V,\lambda_V,J_V)$ be a pair of exact symplectic cobordisms from $Y_0$ to $Y_1$ and $Y_1$ to $Y_2$, respectively, equipped with compatible almost complex structures $J_{R_i}$ on their completions. We denote the boundary inclusions of the contact manifolds $Y_*$ into $U$ and $V$ by $\iota^U_*$ for $* \in \{0,1\}$ and $\iota^V_*$ for $* \in \{1,2\}$.

The neck stretching domain $W_R = U \#_R V$ for parameter $R \in [0,\infty)$ is the exact symplectic cobordism from $(Y_0,e^R\alpha_0)$ to $(Y_2,e^{-R}\alpha_2)$ given by 
\begin{equation}
U \#_R V := U \sqcup_{\iota^U_1} [-R,R] \times Y_1 \sqcup_{\iota^V_1}  V.
\end{equation}
The Liouville forms and complex structures glue to give complex structure $J_R = J_U \#_R J_V$ and $\lambda_R = \lambda_U \#_R \lambda_V$ on the neck stretching domain for each parameter $R$. 

As in Review \ref{review:SFT_compactness}, fix a punctured surface $\Sigma$, and consider a sequence $R_i \in [0,\infty)$ and $u^i:\Sigma \to (\widehat{W}_{R_i}, J_{R_i})$ of $J_{R_i}$--holomorphic curves asymptoting to collections of Reeb orbits $\Gamma^+$ on $Y_0$ and $\Gamma^-$ on $Y_2$, independent of $i$. We remark that the contact forms on the contact boundaries of $(W_R,\lambda_R)$ are equivalent up to multiplication by a scalar, so the Reeb dynamics are independent of $R$ and it is sensible to refer to fixed asymptotics for the curves $u_i$.

The SFT neck stretching theorem provides a topology in which any such sequence $(R_i,u^i)$ converges (after passing to a subsequence) to a holomorphic building $v$ of the form
 \begin{equation}\label{equation:building_notation_neck} v = (u^0_1,\dots,u^0_A,u^U,u^1_1,\dots,u^1_B,u^V,u^2_1,\dots,u^2_C).\end{equation}
Here $A,B,C \in \Z^{\ge 0}$ are integers and the elements of the tuple (called \emph{levels}) are holomorphic maps from punctured surfaces of the form
\[u^*_j:S^*_{j} \to (Y_* \times \R, J_*) \text{ for }* \in \{0,1,2\},\]
\[u^U:S^U \to (\widehat{U},J_U), \qquad \text{ and }\qquad u^V:S^V \to (\widehat{V},J_V). \]
These maps are considered modulo domain reparametrization, and modulo translation when the target manifold is a symplectization. The surfaces $S_{j}$ can be glued together along the boundary punctures asymptotic to matching Reeb orbits, and this glued surface $\#_j S_{j}$ is homeomorphic to $\Sigma$.

The analogous remarks from Review \ref{review:SFT_compactness}, regarding orbit asymptotics and action monotonicity, hold for the building $v$ in (\ref{equation:building_notation_neck}).
\end{review}

\section{Proof of the main result}
\label{section:proof}

In this section, we prove Theorem \ref{theorem:maintheorem} assuming a of technical result, Lemma \ref{lemma:capping_lemma}, which is proven in \S \ref{section:topology}. Here is a brief overview of the proof to help guide the reader. 

We assume by contradiction that the map $\Phi_*$ induced by the family $\Phi$ of (\ref{equation:embedding_family}) is not injective in degree $k$. Using this assumption and the results in \S \ref{section:topology}, we find a certain family of symplectic embeddings, parametrized by a union of an odd number of $k$--tori $\sqcup_1^m T^k$ and built from $\Phi$, which is null--bordant in the space $\op{SympEmp}(E(a),E(b))$. This means that the family extends to a smooth $(k+1)$--dimensional family of symplectic embeddings $\Psi:P\to \op{SympEmp}(E(a),E(b))$ where $P$ is a smooth, compact, $(k+1)$--dimensional manifold with boundary $\partial P \simeq \sqcup_1^m T^k$. 

Using $\Psi$, we construct a moduli space of holomorphic curves $\mathcal{M}_I(\mathfrak{J})$ in completed symplectic cobordisms parametrized by $P$. Moreover, we construct an associated evaluation map $\op{ev}_I:\mathcal{M}_I(\mathfrak{J}) \to T^k$ to a $k$--torus $T^k$. We then show that the degree of this evaluation map is $1 \mod 2$ when restricted to any of the torus components of $\partial \mathcal{M}_I(\mathfrak{J})$. This is the contradiction, since the evaluation map extends to the bounding manifold $\mathcal{M}_I(\mathfrak{J})$ and so must have degree $0$.

\subsection{Moduli spaces in cobordisms} \label{subsection:moduli_spaces}

We now introduce the spaces of holomorphic curves that are relevant to our proof, and derive the salient properties of these spaces. These are generic transversality (Lemma \ref{lem:transversality}), compactness (Lemma \ref{lem:compactness}), and a point count result (Lemma \ref{lem:point_count}).

\begin{notation}[Curve domains] \label{notation:curve_domains} Fix a subset $I \subset \{1,\dots,n\}$ and denote by $|I|$ the size of $I$. For the remainder of \S \ref{section:proof}, we adopt the following notation. 

For each $i \in I$, let $\Sigma_i$ denote a copy of the twice punctured Riemann sphere $\R \times S^1 \simeq \mathbb{CP}^1 \setminus \{0,\infty\}$ with the usual complex structure $j_{\op{\C P^1}}$ and let $\overline{\Sigma}_i$ denote the corresponding copy of $\C P^1$ itself. Let $p^+_i$ and $p^-_i$ denote the points $\infty$ and $0$ in the copy $\overline{\Sigma}_i$ of $\C P^1$. We refer to $p^+_i$ and $p^-_i$ as the positive and negative punctures of $\Sigma_i$, respectively. Denote by $\Sigma_I$ the disjoint union $\sqcup_{i \in I} \Sigma_i$. 
\end{notation}

\begin{definition}[Unparametrized moduli space] \label{definition:unparametrized_moduli_space}
Let $E(a)$ and $E(b)$ be irrational ellipsoids, $W$ be an exact symplectic cobordism  from $\partial E(b)$ to $\partial E(a)$ and $J \in \mathcal{J}(W)$ be an admissible almost complex structure on $W$.

\noindent We define the moduli space $\mathcal{M}_I(J)$ by
\begin{equation}
\mathcal{M}_I(J) := \left \lbrace u:\Sigma_I \to \widehat{W} \; \middle| \; \begin{array}{c}
(du)^{0,1}_{J} = 0 \vspace{5pt}\\

u \to \gamma^\pm_i \text{ at }p^\pm_i\end{array}\right\rbrace\big/(\C^\times)^{|I|}.
\end{equation}

That is, $u:\Sigma_I \to \widehat{W}$ is a $J$--holomorphic curve such that $u$ is asymptotic to the trivial cylinder over $\iota_+(\gamma^+_i)$ in $[0,\infty) \times \partial_+W_{\varphi} \simeq [0,\infty) \times \iota_+(\partial E(b))$ at the puncture $p^+_i$ and $u$ is asymptotic to the trivial cylinder over $\iota_-(\gamma^-_i)$ in $(-\infty,0] \times \partial_-W_{\varphi} \simeq (-\infty,0] \times \iota_-(\partial E(a))$ at the puncture $p^-_i$, for each $i \in I$. We quotient the space of such maps by the group of domain reparametrizations, which is the product $(\C^\times)^{|I|}$ of the biholomorphism groups $\C^\times$ of each component cylinder $\Sigma_i \simeq \R \times S^1$.
\end{definition}

\begin{definition}[Parametrized moduli space over $P$] Let $P$ be a manifold with boundary, $\mathfrak{W}$ be a $P$--parametrized family of exact symplectic cobordisms $\partial E(b) \to \partial E(a)$, and $\mathfrak{J} \in \mathcal{J}(\mathfrak{W})$ be a $P$--parametrized family of complex structure.

\noindent We define the parametric moduli space $\mathcal{M}_I(\mathfrak{J})$ to be the space of pairs
\begin{equation}\label{equation:parametric_moduli_space} \mathcal{M}_I ( \mathfrak{J}) := \left\lbrace (p,u) \; \middle|  \; p \in P,\; u \in \mathcal{M}_I( J_p)\right\rbrace. \end{equation}
\end{definition}

\noindent Our first order of business is establishing transversality for these moduli spaces.

\begin{lemma}[Transversality] 
 \label{lem:transversality}
 Let $E(a)$ and $E(b)$ be irrational symplectic ellipsoids with parameters $a = (a_1,\dots,a_n)$ and $b = (b_1,\dots,b_n)$ satisfying
 \begin{equation}\label{equation:transversality_hypotheses} a_i < b_i, \qquad a_i < 2a_1, \quad\text{and}\quad b_i < 2b_1 \qquad \text{for all $i$ with $1 \le i \le n$.}\end{equation}
 Let $\mathfrak{W}$ be a $P$--parametrized family of exact cobordisms from $\partial E(b)$ to  $\partial E(a)$. Then:
 \begin{itemize}
 \item[(a)] There exists a generic subset $\mathcal{J}^{\op{reg}}(\mathfrak{W}|_{\partial P}) \subset \mathcal{J}(\mathfrak{W}|_{\partial P})$ such that for any $\mathfrak{J}_{\partial P} \in \mathcal{J}^{\op{reg}}(\mathfrak{W}|_{\partial P})$   all $u \in \mathcal{M}_I(J_{\partial P})$ is parametrically Fredholm regular (see \cite[Remark 7.4]{wendl2016} and \cite[Definition 4.5.5]{wendl2010}). The moduli space $\mathcal{M}_I(\mathfrak{J}_{\partial P})$ is then a $(\op{dim}(P) - 1)$--dimensional manifold.
 \item[(b)] Given any $\mathfrak{J}_{\partial P}$ as in (a), there exists a $\mathfrak{J} \in \mathcal{J}^{\op{reg}}(\mathfrak{W})$ such that $\mathfrak{J}|_{\partial P} = \mathfrak{J}_{\partial P}$ and such that every $(p,u) \in \mathcal{M}_I(\mathfrak{J})$ is parametrically Fredholm regular (see \cite[Remark 7.4]{wendl2016} and \cite[Definition 4.5.5]{wendl2010}). In particular, $\mathcal{M}_I(\mathfrak{J})$ is a $\op{dim}(P)$--dimensional manifold with boundary $\partial \mathcal{M}_I(\mathfrak{J}) \simeq \mathcal{M}_I(\mathfrak{J}_{\partial P})$.
\end{itemize}
\end{lemma}

\begin{proof}  This essentially follows from the general transversality results of \cite{wendl2010} and \cite[\S 7]{wendl2016}, which we now discuss in some detail.

First, observe that every curve $u \in \mathcal{M}_I(J_{\partial P})$ must be somewhere injective (see \cite[p. 123]{wendl2016}) for any choice of $J_{\partial P}$. Indeed, note that all of the orbits $\gamma^-_i$ and $\gamma^+_i$ are simple. This means that none of them can be factored as $\eta \circ \varphi$ where $\eta$ is a closed Reeb orbit and $\varphi:S^1 \to S^1$ is a $k$--fold cover with $k \ge 2$. This implies that $u$ is simple as well, i.e. that $u$ cannot factor as $v \circ \phi$ where $v:\Sigma' \to \widehat{W}_{\partial P}$ is a $J_{\partial P}$--holomorphic curve and $\phi:\Sigma_I \to \Sigma'$ is a holomorphic branched cover. Simple curves are somewhere injective. In fact, these conditions are equivalent in our setting, see \cite[Theorem 6.19]{wendl2016}. The same reasoning shows that any curve $u$ appearing as a factor in a point $(p,u) \in \mathcal{M}_I(\mathfrak{J})$ is somewhere injective for any choice of $\mathfrak{J}$.

(Part (a)) we now apply the appropriate parametric version of transversality (see \cite[Theorems 7.1--7.2]{wendl2016}, \cite[Remark 7.4]{wendl2016} and \cite[\S 4.5]{wendl2010}). These results state that there exists a generic (Baire category 2) set $\mathcal{J}^{\op{reg}}(\mathfrak{W}|_{\partial P}) \subset \mathcal{J}(\mathfrak{W}|_{\partial P})$ with the following property: for any $\mathfrak{J}_{\partial P} \in \mathcal{J}^{\op{reg}}(\mathfrak{W}|_{\partial P}) $, any point $(p,u) \in \mathcal{M}_I(\mathfrak{J}_{\partial P})$ where $u$ is somewhere injective is (parametrically) Fredholm regular. In particular, due to the discussion above, every $(p,u) \in \mathcal{M}_I(\mathfrak{J}_{\partial P})$ is Fredholm regular for such a choice of $\mathfrak{J}_{\partial P}$. The dimension of $\mathcal{M}_I(\mathfrak{J}_{\partial P})$ at a point $(p,u)$ is given by the formula 
\begin{equation}\label{equation:index_formula_1} \op{dim}_{(p,u)}( \mathcal{M}_I(\mathfrak{J}_{\partial P})) = \op{dim}(\partial P) + \op{ind}(u),\end{equation}
where $\op{ind}(u)$ denotes the Fredholm index of $u$, given by the following index formula.
\begin{equation}\label{equation:index_formula_2} \op{ind} (u) = \sum_{i \in I} \left((n- 3)\chi(\Sigma_i) + 2c_1(u|_{\Sigma_i},\tau) + \CZ(\gamma^+_i,\tau) - \CZ(\gamma^-_i,\tau)\right).\end{equation} 
The Conley--Zehnder indices $\CZ(\gamma^\pm,\tau)$ and relative Chern numbers $c_1(u|_{\Sigma_i},\tau)$ are as in Review \ref{review:conley_zehnder_indices}, and $\tau$ denotes a trivialization of $\xi$ over $\sqcup_i (\gamma^+_i \sqcup \gamma^-_i)$. 

To simplify the dimension formula, note that $\partial E(a)$ and $\partial E(b)$ are simply-connected and we have assumed that $H_2(W_p;\Z) = 0$. Thus we may choose $\tau$ by taking capping disks $D_i$ for $\gamma^-_i$, thus inducing trivializations of $\xi$ along $\gamma^-_i$, and then extending $\tau$ to a trivialization along $\Sigma_i$ to induce trivializations of $\xi$ along $\gamma^+_i$. The resulting trivialization has $c_1(u|_{\Sigma_i},\tau) = 0$, $\CZ(\gamma^+_i,\tau) = \CZ(\gamma^+_i)$, and $\CZ(\gamma^-_i,\tau) = \CZ(\gamma^-_i)$. Here $\CZ(\gamma^+_i)$ and $\CZ(\gamma^-_i)$ denote the canonical indices described in Review \ref{review:conley_zehnder_indices}. Thus, using this special choice of $\tau$ and noting that $\chi(\Sigma_i)=0$, the formulas (\ref{equation:index_formula_1}-\ref{equation:index_formula_2}) simplify to
\begin{equation}\label{equation:index_formula_3} \op{dim}(\mathcal{M}_I(\mathfrak{J}_{\partial P})) = \op{dim}(P) - 1 + \sum_{i \in I} \left(\CZ(\gamma^+_i) - \CZ(\gamma^-_i)\right).\end{equation}
Finally, we observe that the hypotheses (\ref{equation:transversality_hypotheses}) and the Conley--Zehnder index formula (\ref{equation:CZ_index_for_ellipsoids}) imply that $\CZ(\gamma^+_i) = \CZ(\gamma^-_i) = n - 1 + 2i$. Therefore, the moduli space $\mathcal{M}_I(\mathfrak{J}_{\partial P})$ is $(\op{dim}(P) - 1)$--dimensional, and we have established (a). 

(Part (b)) This part follows from the same parametric transversality results (\cite[Remark 7.4]{wendl2016} and \cite[\S 4.5]{wendl2010}), with the added fact (see \cite[Remark 7.4]{wendl2016} in particular) that we may pick $\mathfrak{J} \in \mathcal{J}(\mathfrak{W})$ to agree with a fixed parametrically transverse $\mathfrak{J}_{\partial P}$ on the boundary. This concludes our discussion for this lemma.\end{proof}

Next we discuss compactness. For the next lemma, we will refer extensively to the review of SFT compactess (Review \ref{review:SFT_compactness}) provided in \S \ref{subsection:holomorphic_curves_and_neck_stretching}.

\begin{lemma}[Compactness]
\label{lem:compactness}  Let $E(a)$ and $E(b)$ be irrational symplectic ellipsoids with parameters $a = (a_1,\dots,a_n)$ and $b = (b_1,\dots,b_n)$. Let $P$ is a compact manifold with boundary, let $\mathfrak{W}$ be a $P$--family of exact symplectic cobordisms with $H_1(W_p) = H_2(W_p) = 0$ and let $\mathfrak{J} \in \mathcal{J}^{\op{reg}}(\mathfrak{W})$ be a family of regular almost complex structures. 

\noindent Then the moduli space $\mathcal{M}_I(\mathfrak{J})$ has the following compactness properties:
\begin{itemize}
	\item[(a)] If $a_n < 2a_1$ and $b_n < 2b_1$, and $P$ has dimension less than or equal to 1, then the moduli space $\mathcal{M}_I(\mathfrak{J})$ is compact.
	\item[(b)] If moreover $a_i < b_i < a_{i+1}$ for all $i\in\{1,\dots,n-1\}$ and $b_n < 2a_1$, then for $P$ of any dimension $\mathcal{M}_I(\mathfrak{J})$ is compact.
\end{itemize}
\end{lemma}

\begin{proof} We prove (a) and (b) by showing that any broken building $u$ arising as a limit of a sequence in $\mathcal{M}_I(\mathfrak{J})$ must consist of a single cobordism level, and thus must be an element of $\mathcal{M}_I(\mathfrak{J})$. Note that the hypotheses of (b) imply those of (a).

Thus let $(p_i,u^i)$ be a sequence in $\mathcal{M}_I(\mathfrak{J})$. Since $P$ is compact, we may pass to a subsequence so that $p_i \to p \in P$. By SFT compactness, after passing to a subsequence $u^i$ converges to a limit building $v$. We use the notation of Review \ref{review:SFT_compactness} for this building. By considering components of $\Sigma_I$ and $\#_j S_{j}$, we can assume that $|I| = 1$, i.e. that $\Sigma_I$ has one component and each $u^i$ is asymptotic to $i$--independent ends $\gamma^+_l$ and $\gamma^-_l$ where $1 \le l \le n$. 

(Part (a)) First, consider a positive symplectization level $u^+_k$. By action monotonicity, we know that $\mathcal{A}[\Gamma^\pm(u^+_k)] \le \mathcal{A}(\gamma^+_l)$. This implies, due to the hypothesis of (a), that each $\Gamma^\pm(u^+_k)$ is a singleton consisting of an embedded orbit of $E(b)$, i.e. that there is a sequence $\{\alpha_k\}_0^M$ such that $\alpha_0 = l$ and 
\[\Gamma^-(u^+_k) = \{\gamma^+_{\alpha_k}\} \quad\text{ for all }k \in \{1,\dots,M\}.\]
Since the genus of the building $v$ must be equal to that of the $u^i$, each $u^+_k$ must therefore be a cylinder from $\gamma^+_{\alpha_{k-1}}$ to $\gamma^+_{\alpha_k}$.

Next, consider the cobordism level $u^W$. Since $\Gamma^+(u^W) = \Gamma^-(u^+_M) = \{\gamma^+_{\alpha_M}\}$, we know by the same discussion as in Lemma \ref{lem:transversality} that $u^W$ is somewhere injective and the parametric moduli space $\mathcal{M}(p,u^W)$ containing $(p,u^W)$ is parametrically Fredholm regular. Therefore we know $\mathcal{M}(p,u^W)$ is a manifold of dimension
\begin{equation} \label{equation:compactness_dim_1}
\op{dim}(\mathcal{M}(p,u^W)) = \op{dim}(P) + \op{ind}(u^{W}) \ge 0.
\end{equation}
Here the index is given by the following formula, similar to (\ref{equation:index_formula_2}):
\[\op{ind}(u) = (n- 3)\chi(S^W) + 2c_1(u^W,\tau) + \sum_{\eta^+ \in \Gamma^+(u^W)} \CZ(\eta^+,\tau) - \sum_{\eta^- \in \Gamma^-(u^W)} \CZ(\eta^-,\tau).\]
As in Lemma \ref{lem:transversality}, the Conley--Zehnder indices $\CZ(\eta^\pm,\tau)$ and relative Chern numbers $c_1(u^W,\tau)$ are as in Review \ref{review:conley_zehnder_indices}, and $\tau$ is a trivialization of $\xi$ over $\Gamma^+(u^W) \sqcup \Gamma^-(u^W)$. 

Now we simplify this dimension formula. Due to the assumption that $H_1(W_p) = H_2(W_p)$ for each fiber $W_p$ of $\mathfrak{W}$, we can choose a trivialization $\tau$ extending over $u^W$ such that $\op{CZ}(\eta^\pm,\tau) = \op{CZ}(\eta^\pm)$ and $c_1(u^W,\tau) = 0$. Furthermore, we must have $g(S^W) = 0$ since otherwise the total genus of $v$ would be greater than $0$. This implies that $\chi(S^W) = 1 - |\Gamma^-(u^W)|$. Thus the index formula simplifies to the following.
\begin{equation} \label{equation:compactness_index_formula2}
\op{dim}(\mathcal{M}(p,u^W)) =\op{dim}(P) + (n-3) + \op{CZ}(\gamma^+_{\alpha_M}) - \sum_{\eta^- \in \Gamma^-(u^W)} \left((n-3) + \op{CZ}(\eta^-)\right)
\end{equation}Applying the hypothesis (a) and the CZ index formula (\ref{equation:CZ_index_for_ellipsoids}) in Example \ref{example:ellipsoid_dynamics}, we note that  
\[(n-3) + \op{CZ}(\gamma^\pm_i) = 2n - 4 + 2i \quad\text{and}\quad (n-3) + \op{CZ}((\gamma^\pm_i)^m) \ge 4n - 2 \]
for any simple orbit $\gamma^\pm_i$ and any $m \ge 2$.

Since $\op{dim}(P) \le 1$, this implies that (\ref{equation:compactness_index_formula2}) is negative if either $|\Gamma^-(u^W)| \ge 2$, or $\eta^- = (\gamma^-_i)^m$ for $m \ge 2$ and some $i$, or if $\eta^- = \gamma^-_{m_0}$ for $m_0 > \alpha_M$ (where $\gamma^+_{\alpha_M}$ is the positive end of $u^W)$). Thus we must have $\Gamma^-(u^W) = \{\gamma^-_{\beta_0}\}$ for some $\beta_0 \le \alpha_M$.

Finally, we can argue analogously to the positive symplectization case to show that there is a sequence $\{\beta_k\}_0^N$ such that $\beta_N = l$ and 
\[\Gamma^+(u^-_k) = \{\gamma^-_{\beta_k}\} \quad\text{ for all }k \in \{1,\dots,N\}.\]
We have thus shown that every level of $v$ is Fredholm regular and cylindrical. Therefore we have the following inequalities of the Conley--Zehnder indices.
\[
\op{CZ}(\gamma^-_l) = \op{CZ}(\gamma^-_{\beta_N}) \le \dots \le \op{CZ}(\gamma^-_{\beta_0}) \le \op{CZ}(\gamma^+_{\alpha_M}) \le \dots \le \op{CZ}(\gamma^+_{\alpha_0}) = \op{CZ}(\gamma^+_l).
\]
Since $\op{CZ}(\gamma^-_l) = \op{CZ}(\gamma^+_l)$, we thus conclude that every symplectization level $u^\pm_k$ is index $0$. This implies that they are somewhere injective branched covers of trivial cylinders, which must in fact be trivial cylinders. This is only possible if $M = N = 0$ and thus $v$ only has a cobordism level.

(Part (b)) Begin by considering a positive symplectization level $u^+_j$ of $v$. Due to action monotonicity (\ref{equation:action_monotonicity_b}), the collections $\Gamma^+(u^+_j)$ and $\Gamma^-(u^+_j)$ of positive and negative limit Reeb orbits of $(Y_+,\alpha_+) \simeq (\partial E(b),\lambda|_{\partial E(b)})$ must satisfy
\[a_{l_i}= \mathcal{A}(\gamma^-_{l_i}) \le \mathcal{A}[\Gamma^-(u^+_j)] \le \mathcal{A}[\Gamma^+(u^+_j)] \le \mathcal{A}(\gamma^+_{l_i}) = b_{l_i}.\]
Consider $\Gamma^+(u^+_j)$ only. Due to the hypotheses of (b), $\Gamma^+(u^+_j)$ cannot contain either a copy of $\gamma^+_r$ for $r > l_i$ or a copy of an iterate $(\gamma^+_r)^m$ for any $m \ge 2$ and any $r$. Otherwise, we would have $\mathcal{A}[\Gamma^+(u^+_j)] > b_{l_i}$. This implies that $\Gamma^+(u^+_j)$ can only contain Reeb orbits $\gamma^+_r$ for $r \le l_i$. Moreover, since $\mathcal{A}[\Gamma^+(u^+_j)] \ge a_{l_i}$ and $\mathcal{A}(\gamma^+_r) = b_r < a_{l_i}$ for $r < l_i$ (again by (b)), we must have $\Gamma^+(u^+_j) = \{\gamma^+_{l_i}\}$.  The same reasoning shows that $\Gamma^-(u^+_j) = \{\gamma^+_{l_i}\}$. 

This demonstrates that the energy $\mathcal{E}(u^+_j)$ of the level $u^+_j$ is $0$ by (\ref{equation:energy_stokes_symp}) and the level $u^+_j$ must be a branched cover of a trivial cylinder (see \cite[Lemma 9.9]{wendl2010}). Since the ends are embedded, $u^+_j$ must be simple and thus a trivial cylinder. This is disallowed by the SFT compactness statement, so $u^+_j$ cannot exist.

The same reasoning implies that negative levels $u^-_j$ of $v$ cannot exist. Thus the building $v$ consists of only a cobordism level $u^W$. \end{proof}

Finally, we state and prove the following curve count result, Lemma \ref{lem:point_count}. For this proof, we will use SFT neck stretching as discussed in Review \ref{review:SFT_neck_stretching}.

\begin{lemma}[Curve count]
 \label{lem:point_count} Let $E(a)$ and $E(b)$ be irrational symplectic ellipsoids with parameters $a = (a_1,\dots,a_n)$ and $b = (b_1,\dots,b_n)$ satisfying
\begin{equation} \label{equation:point_count_hypotheses} 
a_i < b_i < a_{i+1} \text{ for all } i\in\{1,\dots,n-1\} \quad \text{and} \quad b_n < 2a_1.
\end{equation}
Let $\varphi:E(a) \to E(b)$ be a symplectic embedding which is isotopic to the inclusion $\iota:E(a) \to E(b)$. Finally, let $W_\varphi$ be the symplectic cobordism associated to $\varphi$ and let $J \in \mathcal{J}^{\op{reg}}(W_\varphi)$ be a regular almost complex structure (provided by Lemma \ref{lem:transversality}).  
Then the number of points $\# \mathcal{M}_I(J)$ in the moduli space $\mathcal{M}_I(J)$ is odd.
\end{lemma}

\begin{remark}[Floer theoretic proof] Morally, one may view the holomorphic cylinders in $\mathcal{M}_I(\mathfrak{J})$ as contributing to the cobordism map $CH(W_\varphi):CH(\partial E(b)) \to CH(\partial E(a))$ (where one can take $CH$ to be either the full contact homology or perhaps cylindrical contact homology). The invariance of the signed or mod $2$ point count of $\mathcal{M}_I(J)$ may be viewed essentially as a consequence of the deformation invariance of this cobordism map. A proof of Lemma \ref{lem:point_count} in this spirit is possible using the foundations from e.g. \cite{pardon2015b}. Here we provide a simpler argument which does not use Floer theory directly.
\end{remark}

\begin{proof} We first address the case where $\varphi = \iota$ and then tackle the general case.

(Case of $\varphi = \iota$) Pick an $\epsilon > 0$ sufficiently small so that we have an inclusion $\jmath:E(\epsilon \cdot b) \to E(a)$, and let $W_\iota,W_\jmath$ and $W_{\iota \circ \jmath}$ be the cobordisms associated to $\iota,\jmath$ and $\iota \circ \jmath$ respectively. Pick regular $J_\iota \in \mathcal{J}^{\op{reg}}(W_\iota)$ and $J_\jmath \in \mathcal{J}^{\op{reg}}(W_\jmath)$ so that (by  Lemmas \ref{lem:transversality}(b) and \ref{lem:compactness}(a)), the spaces $\mathcal{M}_I(J_\iota)$ and $\mathcal{M}_I(J_\jmath)$ are compact $0$--dimensional manifolds with Fredholm regular points. 

Note that $\widehat{W}_{\iota \circ \jmath}$ is equivalent (as a completed cobordism) to the symplectization of $\partial E(b)$. Thus we can choose  $J_{\iota \circ \jmath} \in \mathcal{J}^{\op{reg}}(W_{\iota \circ \jmath})$  to be translation invariant and such that the moduli spaces $\mathcal{M}_I(J_{\iota \circ \jmath})$ are all transverse (due to the same arguments as in Lemma \ref{lem:compactness}, but using the transversality theorem \cite[Theorem 8.1]{wendl2010} for symplectizations). Any $J_{\iota \circ \jmath}$--holomorphic cylinder $u$ from $\gamma^+_i$ to $\gamma^+_i$ with index $0$  must be translation invariant (since otherwise the dimension of the moduli space would be positive) and embedded, thus a trivial cylinder. Thus $\mathcal{M}_I(J_{\iota \circ \jmath})$ consists of a single point $\{u\}$, which is a product of trivial cylinders. 

Now note that we may write $W_{\iota \circ \jmath} = E(b) \setminus E(\epsilon \cdot b)$ as a composition of cobordisms $W_{\iota \circ \jmath} = W_\iota \# W_\jmath$ and consider the neck stretching domain $W_{\iota \circ \jmath}^R = W_\iota \#_R W_\jmath$ for each $R$. This yields a $[0,\infty)$--parametrized family of cobordisms $\mathfrak{W}$ whose fiber at $R \in [0,\infty)$ is $W_{\iota \circ \jmath}^R$. Let $\mathfrak{J} \in \mathcal{J}(\mathfrak{W})$ be an almost complex structure which agrees with the glued structure $J_R := J_\iota \#_R J_\jmath$ for $R$ near $\infty$, and consider the moduli space 
\begin{equation} \label{equation:neck_stretch_moduli_space} \mathcal{M}_I(\mathfrak{J}) := \{(R,u)\; |\;R \in [0,\infty) \text{ and }u \in M_I(J_R)\}.\end{equation} 
For sufficiently large $R_0$, there exists a gluing map of the form
\[
\op{glue}:\mathcal{M}_I(J_\iota) \times \mathcal{M}_I(J_\jmath) \times (R_0,\infty) \to \mathcal{M}_I(\mathfrak{J}), \qquad (u_\iota,u_\jmath,R) \mapsto (R,\op{glue}_R(u_\iota,u_\jmath)).
\]
This map is constructed in (for instance) \cite[\S 5]{pardon2015b}, and is a homeomorphism onto its image for $R_0$ sufficiently large. Using Lemma \ref{lem:transversality} and the references therein, we may choose $\mathfrak{J}$ to be parametrically regular over the region $[0,R_0]$ of $P$ and to agree with $J$ and $J_\iota \#_R J_\jmath$ at $0$ and $R \in [R_0,\infty)$. We may form a compactified moduli space
\begin{equation} \label{equation:neck_stretch_moduli_space_compactified}
\overline{\mathcal{M}}_I(\mathfrak{J}) := \mathcal{M}_I(\mathfrak{J}) \sqcup_{\op{glue}} \mathcal{M}_I(J_\iota) \times \mathcal{M}_I(J_\jmath) \times (R_0,\infty].
\end{equation}
This moduli space is a compact $1$--manifold with boundary
\begin{equation}
\partial \overline{\mathcal{M}}_I(\mathfrak{J}) = \mathcal{M}_I(J_{\iota\circ\jmath}) \sqcup \mathcal{M}_I(J_\iota) \times \mathcal{M}_I(J_\jmath).
\end{equation}
The number of boundary points of a compact $1$--manifold is even. In particular
\[
\#\mathcal{M}_I(J_{\iota\circ\jmath}) + \#\mathcal{M}_I(J_\iota) \cdot \# \mathcal{M}_I(J_\jmath) = 0 \mod 2.
\]
Since $\mathcal{M}_I(J_{\iota\circ\jmath})$ has an odd number of points, it follows that $\mathcal{M}_I(J_\iota)$ and $\mathcal{M}_I(J_\jmath)$ do to.

(General Case) Let $J$ be as chosen in the lemma statement, and let $\Phi:E(a) \times [0,1] \to E(b)$ be an isotopy of symplectic embeddings with $\Phi_0 = \iota$ and $\Phi_1 = \varphi$. Let $\mathfrak{W}$ denote the $[0,1]$--parametrized family of exact symplectic cobordisms from $\partial E(b)$ to $\partial E(a)$ with fiber $W_{\Phi_t}$ at $t \in [0,1]$. Pick any $J_\iota \in \mathcal{J}^{\op{reg}}(W_\iota)$ as permitted by Lemma \ref{lem:transversality}(b). Then by Lemma \ref{lem:transversality}(b) and \ref{lem:compactness}(a), we may choose a regular $\mathfrak{J} = \{J_t\}_{t\in[0,1]} \in \mathcal{J}^{\op{reg}}(\mathfrak{W})$ such the parametric moduli space $M_I(\mathfrak{J})$ over $[0,1]$ is Fredholm regular, and such that $J_0 = J_\iota$ and $J_1 = J$. It follows that
\[
\partial \mathcal{M}_I(\mathfrak{J}) = \mathcal{M}_I(J) \sqcup \mathcal{M}_I(J_\iota) \quad\text{and}\quad \# \mathcal{M}_I(J) = \# \mathcal{M}_I(J_\iota) \mod 2.
\]
Thus the general case follows from the case where $\varphi = \iota$.
\end{proof}

\begin{lemma} The compactification $\overline{\mathcal{M}}_I(\mathfrak{J})$ (see (\ref{equation:neck_stretch_moduli_space_compactified})) of the moduli space $\mathcal{M}_I(\mathfrak{J})$ (see (\ref{equation:neck_stretch_moduli_space})) is compact in the neck stretching topology discussed in Review \ref{review:SFT_neck_stretching}.
\end{lemma}
 
\begin{proof} It suffices to take a sequence $(R_i,u^i) \in \mathcal{M}_I(\mathfrak{J})$ and show that it has a convergent subsequence in $\overline{\mathcal{M}}_I(\mathfrak{J})$. If $R_i \le C$ for some fixed upper bound $C$, then $(R_i,u^i)$ is a sequence of holomorphic curves in cobordisms parametrized over a compact, $1$--dimensional space $[0,C]$ and we can apply Lemma \ref{lem:compactness}. Thus we may assume that $R_i \to \infty$ as $i \to \infty$. Let $v$ denote the limit building provided by Review \ref{review:SFT_neck_stretching}.

Let $\gamma^b_i, \gamma^a_i$, and $\gamma^{\epsilon b}_i$ denote the simple orbits on $\partial E(b), \partial E(a)$, and $\partial E(\epsilon b)$ respectively, ordered by increasing action as usual. We denote the levels of $v$ by
\[
v = (u^b_1,\dots,u^b_A,u^\iota,u^a_1,\dots,u^a_B,u^\jmath,u^{\epsilon b}_1,\dots,u^{\epsilon b}_C).
\]
By considering components, we may assume that $|I| = 1$ and $u_i$ is asymptotic to $\gamma^b_l$ at the positive end and $\gamma^{\epsilon b}_l$ on the negative end for some $l$. An identical argument to that in Lemma \ref{lem:compactness}(a) shows that every level of $v$ is transverse, embedded, genus $0$ and asymptotic to a single embedded orbit at the positive end and a single embedded orbit at the negative end. 

The remainder of the proof is also like Lemma \ref{lem:compactness}(a). Since the index of each level must be nonnegative by transversality, the index of the limit orbits must be nondecreasing. Then $\op{CZ}(\gamma^b_i) = \op{CZ}(\gamma^{\epsilon b}_i)$ implies that every orbit has the same index, and that all of the levels are index $0$. In particular, any symplectization level must be a trivial cylinder, and thus can't exist. This implies the result.
\end{proof}

\subsection{Proofs of Theorem \ref{theorem:maintheorem} and Proposition \ref{prop:smallellipsoid}}

In this section, we use the moduli spaces constructed in \S \ref{subsection:moduli_spaces} to prove our main result, Theorem \ref{theorem:maintheorem}. We also provide a proof of Proposition \ref{prop:smallellipsoid}. The following small piece of notation will be helpful for both proofs. 

\begin{notation} For any $n \in \Z^+$ and any $I \subset \{1, \dots, n\}$, define the $|I|$--torus by
\begin{equation}
T_I = \left \lbrace  (\theta_1, \dots, \theta_n) \in T^n \; \middle| \; \theta_j = 0, \;\forall j\notin I \right \rbrace.
\end{equation}
Note that the $k$th homology group $H_k(T^n;\Z/2)$ of the $n$--torus $T^n$ is generated by the fundamental classes $[T_I]$, where $I$ runs over all subsets of size $|I| = k$. 
\end{notation}

For Theorem \ref{theorem:maintheorem}, we also require the following result, which is proven in \S \ref{section:topology}. 

\begin{lemma} \label{lemma:capping_lemma} Let $U$ and $V$ be compact symplectic manifolds with boundary. Let $Z$ be a closed manifold with total Stieffel--Whitney class $w(Z) = 1 \in H^*(Z;\Z/2)$ and let $\Phi$ be a smooth family of symplectic embeddings
\[\Phi:Z \to \op{SympEmb}(U,V) \quad \text{with}\quad \Phi_*[Z] = 0.\]
Then there exists a compact manifold $P$ with boundary $Z$ and an extension of $\Phi$ to a smooth family $\Psi$ of symplectic embeddings
\[\Psi:P \to \op{SympEmb}(U,V) \quad \text{with}\quad \Psi|_{\partial P} = \Phi.\]
\end{lemma}

Given the above preparation, we are now ready for the proof of Theorem \ref{theorem:maintheorem}.

\begin{proof}(Theorem \ref{theorem:maintheorem})
We pursue the argument by contradiction outlined at the begining of \S \ref{section:proof}. Fix an integer $k$ with $1 \le k \le n$ and suppose that there were a nonzero $\Z/2$--homology class of the $n$--torus of the form
\begin{equation}\label{equation:torus_homology_class} [A] = \sum_L c_L [T_L] \quad\text{with}\quad\Phi_* [A]=0 \in H_k(\op{SympEmb}(E(a),E(b));\Z/2).\end{equation}

Let $Z = \sqcup_{c_L\neq 0} T_L$. Then Lemma \ref{lemma:capping_lemma} states that there exists a smooth $(k+1)$--dimensional manifold $P$ with boundary $\partial P = Z$ and a smooth family of embeddings
 \begin{equation} \label{equation:extended_family} \Psi: P\to \op{SympEmb}(E(a),E(b)) \quad\text{with}\quad \Psi|_{T_L} = \Phi|_{T_L} \text{ for each }T_L \subset Z. \end{equation} 

 By passing to subellipsoids, we may assume that $E(a)$ and $E(b)$ are irrational. In this setting, Lemmas \ref{lem:transversality} and \ref{lem:compactness} state that there exist choices of $J_{\partial P} \in \mathcal{J}(\widehat{W}_{\partial P})$ and $\mathfrak{J} \in \mathcal{J}(\Psi)$ such that the parametrized moduli space $\mathcal{M}_{I}({\mathfrak{J}})$ is a compact $(k+1)$--dimensional manifold with boundary $ \partial \mathcal{M}(\mathfrak{J}) \simeq Z \times \mathcal{M}(\widehat{W}_{\partial P}; J_{\partial P})$. On the parametrized moduli space $\mathcal{M}_I(\mathfrak{J})$, we can define an evaluation map 

\begin{equation} \label{equation:evaluation_map_definition} \op{Ev}_{I}: \prod_{i \in I} \gamma^+_i \times \mathcal{M}_{I}(\mathfrak{J}) \to \prod_{i \in I} \gamma^-_i  \simeq T^k \end{equation}
via the following procedure. Let $q = (q_i)_{i \in I}$ be a point in $\times_{i \in I} \gamma^+_i$ and let $(p,u) \in \mathcal{M}_{I}(\mathfrak{J})$. According to (\ref{equation:parametric_moduli_space}), $(p,u)$ is a pair of a point $p \in P$ and an equivalence class of holomorphic maps $u:\Sigma_I \to \widehat{W}_{\Psi_p}$ up to reparametrization. Pick a representative holomorphic curve $\widetilde{u}$ of $u$, which consists of $k$ maps $\widetilde{u}_i:\Sigma_i \to \widehat{W}_{\Psi_p}$ for each $i \in I$. We have limit parametrizations of $\gamma^+_i$ and $\gamma^-_i$ induced by $\widetilde{u}_i$, defined by
\[
\lim_+\widetilde{u}_i:S^1 \to \gamma^+_i \subset \partial E(b), \qquad \quad \lim_+\widetilde{u}_i(t) := \lim_{s \to + \infty} \pi_{\partial_+W_{\Psi_p}} \widetilde{u}_i(s,t),
\]
\[
\lim_-\widetilde{u}_i:S^1 \to \Psi_p(\gamma^-_i) \subset \Psi_p(\partial E(a)), \qquad \lim_-\widetilde{u}_i(t) := \lim_{s \to - \infty} \pi_{\partial_-W_{\Psi_p}} \widetilde{u}_i(s,t).
\]
Here $\pi_{\partial_+W_{\Psi_p}}$ and $\pi_{\partial_-W_{\Psi_p}}$ denote projection to the positive and negative boundaries of $W_{\Psi_p}$. Note that these projections are only defined in the limit as $s \to \pm \infty$. In terms of these parametrizations, we define the evaluation map $\op{Ev}_I$ by the formula
\begin{equation} \label{equation:evaluation_map_formula}
\op{Ev}_I(q;p,u) := \left([\Psi_p^{-1} \circ \lim_- \widetilde{u}_i \circ (\lim_+ \widetilde{u}_i)^{-1}] (q_i)\right)_{i \in I} \in \prod_{i \in I} \gamma^-_i.
\end{equation}This definition is independent of the choice of representative $\widetilde{u}$. Finally, fix an arbitrary $\underline{q}$ in the product $\prod_{i \in I} \gamma^+_i$ and define 
\begin{equation}\label{equation:small_evaluation_map}\op{ev}_I:\mathcal{M}_{I}(\mathfrak{J}) \to \prod_{i \in I} \gamma^-_i, \qquad \op{ev}_I(p,u) := \op{Ev}_I(\underline{q};p,u).\end{equation}

Now consider the restriction of $\op{ev}_I$ to each component $T_L \times \{u\}$ of the boundary $Z \times \mathcal{M}(\widehat{W}_{\partial P},J_{\partial P})$ of $\mathcal{M}(\mathfrak{J})$. Since the equivalence class of curve $u$ is independent of $\theta \in T_L \subset T^n$, we can use (\ref{equation:evaluation_map_formula}) and (\ref{equation:small_evaluation_map}) to write
\[\op{ev}_I(\theta,u) = (\Psi_\theta^{-1}(\underline{r}_i))_{i \in I} \qquad \text{with}\quad\underline{r} := \left([\lim_- \widetilde{u}_i \circ (\lim_+ \widetilde{u}_i)^{-1}] (\underline{q}_i)\right)_{i \in I} \in \prod_{i \in I} \gamma^-_i.\]
Here $\underline{r}$ is independent of $\theta$. Using the fact that $\Psi^{-1}|_Z = \Phi^{-1}$ and the formula (\ref{equation:embedding_family}) for the family of embeddings $\Phi$, we have the formula
\begin{equation} \label{equation:degree_calculation}
\op{ev}_I(\theta,u) = (\Phi_\theta^{-1}(\underline{r}_i))_{i \in I} = (e^{-2\pi i\theta_i} \cdot \underline{r}_i)_{i \in I}.
\end{equation}
In the right--most expression of (\ref{equation:degree_calculation}), we identify $\underline{r}_i$ with an element of $\C^n$ via the inclusion $\gamma_i^- \subset E(a) \subset \C^n$. 

The expression (\ref{equation:degree_calculation}) allows us to compute the degree of $\op{ev}_I$ on each component $T_L \times \{u\}$. There are two cases. If $L = I$, then (\ref{equation:degree_calculation}) shows that $\op{ev}_I|_{T_L \times \{u\}}$ is degree $1$. If $I \neq L$, then for any $j \in L \setminus (L \cap I)$, $\theta_j$ is constant for every $\theta \in T_L$ and it follows from  (\ref{equation:degree_calculation}) that the degree of $\op{ev}_I|_{T_L \times \{u\}}$ is $0$. To derive our final contradictiction, we now observe that the total degree mod $2$ of $\op{ev}_I$ restricted to the boundary is
\begin{equation} 
\label{equation:final_point_count}
\begin{split}
\op{deg}(\op{ev}_I|_{\partial \mathcal{M}_I(\mathfrak{J})}) &= \sum_{T_L \times \{u\} \subset \partial \mathcal{M}_I(\mathfrak{J})} \op{deg}(\op{ev}_I|_{T_L \times \{u\}}) = \\
&=|\mathcal{M}_I(J_{\partial P})| \equiv 1 \mod 2.
\end{split}
\end{equation}
The right--most equality in (\ref{equation:final_point_count}) crucially uses the point count of Lemma \ref{lem:point_count}. The equality (\ref{equation:final_point_count}) also provides the contradiction, since the degree of the restricton of a map to a boundary must be $0$ mod $2$. This concludes the proof.
\end{proof}

Having concluded the proof of Theorem \ref{theorem:maintheorem}, we now move on to Proposition \ref{prop:smallellipsoid}. The proof is much less involved than that of Theorem \ref{theorem:maintheorem}, and does not use any of the machinery from \S \ref{subsection:review_of_contact_geometry}--\ref{subsection:moduli_spaces}. We begin with a lemma about the homology groups of the unitary group $U(n)$.

\begin{lemma} \label{lem:unitary_contractible} Consider the map $U:T^n \to U(n)$ given by $\theta \mapsto U_\theta$. Then the induced map $U_*:H_*(T^n;\Z/2) \to H_*(U(n);\Z/2)$ on $\Z/2$--homology is:
\begin{itemize}
\item[(a)] surjective if $* = 0$ or $* = 1$.
\item[(b)] identically $0$ if $* \ge 2$.
\end{itemize}
\end{lemma}

\begin{proof} To show (a), we first note that $T^n$ and $U(n)$ are connected so $U_*|_{H_0} = \op{Id}$. Furthermore, if we consider the loop $\gamma:\R/2\pi \Z \to T^n$ given by $\theta \mapsto (\theta,0,\dots,0)$, we see that the composition
\[
\op{det}_\C \circ\ U \circ \gamma:\R/2\pi \Z \to U(1) \simeq \R/2\pi \Z
\]
is the identity. Since $\op{det}_\C:U(n) \to U(1)$ induces an isomorphism on $H_1$, the induced map of $U$ must be surjective on $H_1$.

To show (b) we proceed as follows. It suffices to show that $U_*[T_L] = 0$ for all $L$ with $|L| \ge 2$. We can factorize $T_L = T_J \times T_K$ for $J \sqcup K = L$ and $|J| = 2$, and 
\[
T_L = T_J \times T_K \xrightarrow{\iota_J \times \iota_K} U(2) \times U(n-2) \xrightarrow{j} U(n).
\]
Here $j$ is the inclusion of a product of unitary subgroups, and $\iota_J$ and $\iota_K$ are inclusions of the tori into these unitary subgroups. It suffices to show that $(\iota_J \times \iota_K)_*[T_L] = [\iota_J]_*[T_J] \otimes [\iota_K]_*[T_K] = 0$, or simply that $[\iota_J]_*[T_J] = 0 \in H_2(U(2);\Z/2)$. 

Now we simply note that $\op{dim}(U(2)) = 4$ and $H^*(U(2);\Z/2) \simeq \Z/2[c_1,c_3]$ where $c_i$ is a generator of index $i$. In particular, $H_2(U(2);\Z/2) \simeq H^2(U(2);\Z/2) = 0$.\end{proof}

Using Lemma \ref{lem:unitary_contractible}, we can now prove Propositon \ref{prop:smallellipsoid}. The point is that the entire unitary group $U(n)$ embeds into $\op{SympEmb}(E(a),E(b))$ via domain restriction when $a_i < b_j$ for all $i$ and $j$ (which is equivalent to $a_n < b_1$ by our ordering convention).  

\begin{proof} (Proposition \ref{prop:smallellipsoid}) Let $D:\op{SympEmb}(E(a),E(b)) \to U(n)$ denote the map $\varphi \mapsto r(d\varphi|_0)$, given by taking derivatives $d\varphi|_0 \in \op{Sp}(2n)$ at the origin and composing with a retraction $r:\op{Sp}(2n) \to U(n)$. Under the hypotheses on $a$ and $b$, we can factor the identity $\op{Id}:U(n) \to U(n)$ and $\Phi:T^n \to \op{SympEmb}(E(a),E(b))$ as 
\[
\op{Id}:U(n) \xrightarrow{\op{res}} \op{Symp}(E(a),E(b)) \xrightarrow{D} U(n),
\]
\[
\Phi:T^n \xrightarrow{U} U(n) \xrightarrow{\op{res}} \op{SympEmb}(E(a),E(b)).
\]
Here $\op{res}(\varphi) := \varphi|_{E(a)}$ denotes restriction of domain. In particular, $\op{res}:U(n) \to \op{Symp}(E(a),E(b))$ is injective on homology and $\op{Im}(\Phi_*) \simeq \op{Im}(U_*)$ as $\Z$--graded $\Z/2$--vector spaces. The result thus follows from Lemma \ref{lem:unitary_contractible}.
\end{proof}

\section{Spaces of symplectic embeddings} 
\label{section:topology}

In this section, we discuss some basic results about the Fr\'echet manifold of symplectic embeddings $\op{SympEmb}(U,V)$ between symplectic manifolds with boundary. In \S \ref{subsection:Frechet_manifold_structure}, we construct the Fr\'echet manifold structure on $\op{SympEmb}(U,V)$. In \S \ref{subsection:bordism_groups}, we discuss the relationship between the bordism groups and homology groups of a Fr\'echet manifold. Last, we prove a version of the Weinstein neighborhood with boundary as Proposotion \ref{prop:weinstein_nbhd} in \S \ref{subsection:weinstein_neighborhood_w_bdry}.

\subsection{Fr\'echet manifold structure} \label{subsection:Frechet_manifold_structure} Let $(U,\omega_U)$ and $(V,\omega_V)$ be $2n$--dimensional compact symplectic manifolds with nonempty contact boundaries. We now give a proof of the folklore result that the space of symplectic embeddings from $U$ to $V$ is a Fr\'echet manifold. 

\begin{proposition} \label{prop:Frechet_structure_on_symp} The space $\op{SympEmb}(U,V)$ of symplectic embeddings $\varphi:U \to \op{int}(V)$ with the $C^\infty$ compact open topology is a metrizable Fr\'echet manifold.
\end{proposition}

\begin{proof} Let $(U \times V, \omega_{U \times V})$, with $\omega_{U \times V} = \pi^*_U\omega_U - \pi^*_V\omega_V$, denote the product symplectic manifold with corners. Given a symplectic embedding $\varphi:U \to \op{int}(V)$, we may associate the graph $\Gamma(\varphi) \subset U \times V$ given by
\[
\Gamma(\varphi) := \{(u,\varphi(u)) \in U \times V\}.
\]
The graph is a Lagrangian submanifold with boundary transverse to the characteristic foliation $T(\partial U)^\omega$ on the contact hypersurface $\partial U \times \op{int}(V)$. By the Weinstein neighborhood theorem with boundary, Proposition \ref{prop:weinstein_nbhd}, there is a neighborhood $A$ of $U$, a neighborhood $B$ of $\Gamma(\varphi)$ and a symplectomorphism $\psi:A \simeq B$ with $\psi|_U:U \to \Gamma(\varphi)$ given by $u \mapsto (u,\varphi(u))$ and $\psi^*\omega_{U \times V} = \omega_{\op{std}}$. 

Let $\mathcal{A}(\varphi,\psi) \subset \op{ker}(d:\Omega^1(L) \to \Omega^2(L))$ and $\mathcal{B}(\varphi,\psi) \subset \op{SympEmb}(U,V)$ denote the open subsets given by
\[
\mathcal{A}(\varphi,\psi) := \left\lbrace \alpha \in \Omega^1(L)\; \middle| \; d\alpha = 0 \text{ and}\op{Im}(\alpha) \subset A\right\rbrace,
\]
\[
\mathcal{B}(\varphi,\psi) := \left\lbrace \phi \in \op{SympEmb}(U,V)\; \middle| \; \op{Im}(\varphi) \subset B\right\rbrace.
\]
Then we have maps $\Phi:\mathcal{A}(\varphi,\psi) \to \mathcal{B}(\varphi,\psi)$ and $\Psi:\mathcal{A}(\varphi,\psi) \to \mathcal{B}(\varphi,\psi)$ given by
\[
\alpha \mapsto \Phi[\alpha] := (\pi_V \circ \psi \circ \alpha) \circ (\pi_U \circ \psi \circ \alpha)^{-1},
\]
\[
\phi \mapsto \Psi[\phi] := (\psi^{-1} \circ (\op{Id} \times \phi)) \circ (\pi_L \circ \psi^{-1} \circ (\op{Id} \times \phi))^{-1}.
\]
It is a tedious but straightforward calculation to check that $\Phi \circ \Psi = \op{Id}$ and $\Psi \circ \Phi = \op{Id}$. The fact that $\Phi$ and $\Psi$ are continuous in the $C^\infty$ compact open topologies on the domain and images follows from the fact that function composition defines a continuous map $C^\infty(M,N) \times C^\infty(N,O) \to C^\infty(M,O)$ for any compact manifolds $M$, $N$, and $O$ (in fact, smooth; see \cite[Theorem 42.13]{kriegl1997}).

Since $C^\infty(U,V)$ is metrizable under the compact open $C^\infty$--topology (see \cite[Corollary 41.12]{kriegl1997}), the subspace $\op{SympEmb}(U,V)$ is also metrizable. \end{proof}

\begin{lemma} Let $L$ be a compact manifold with boundary and let $\sigma:L \to T^*L$ be a section. Then $\sigma(L)$ is Lagrangian if and only if $\sigma$ is closed.
\end{lemma}

\begin{proof} The same as the closed case, see \cite[Proposition 3.4.2]{mcduff2017}.
\end{proof}

\subsection{Bordism groups of Fr\'echet manifolds} \label{subsection:bordism_groups} We now discuss (unoriented) bordism groups and their structure in the case of Fr\'echet maifolds. We begin by defining the relevant notions of (continuous and smooth) bordism. 

\begin{definition}[Bordisms]Let $X$ be a topological space and $f:Z \to X$ be a map from a closed manifold. We say that the pair $(Z,f)$ is \emph{null--bordant} if there exists a pair $(Y,g)$ of a compact manifold with boundary $Y$ and a continuous map $g:Y \to X$ such that $\partial Y = Z$ and $g|_{\partial Y} = f$. Given a pair of manifold/map pairs $(Z_i,f_i)$ for $i \in \{0,1\}$, we say that $(Z_0,f_0)$ and $(Z_1,f_1)$ are \emph{bordant} if $(Z_0 \sqcup Z_1,f_0 \sqcup f_1)$ is null--bordant. 
\end{definition}

\begin{definition}[Smooth bordism] Let $X$ be a Fr\'echet manifold and $f:Z \to X$ be a smooth map from a smooth closed manifold. Then $(Z,f)$ is \emph{smoothly null--bordant} if it is null--bordant via a pair $(Y,g)$ where $g:Y \to X$ be a smooth map of Banach manifolds with boundary. Similarly, a pair $(Z_i,f_i)$ for $i \in \{0,1\}$ is \emph{smoothly bordant} if $(Z_0 \sqcup Z_1, f_0 \sqcup Z_1)$ is smoothly null--bordant.
\end{definition}

The above notions come with accompanying versions of the bordism group.

\begin{definition}[Bordism group of $X$] The \emph{$n$--th bordism group} $\Omega_n(X;\Z/2)$ of a topological space $X$ is group generated by equivalence classes $[Z,f]$ of pairs $(Z,f)$, where $Z$ is a closed $n$--dimensional manifold and $f:Z \to X$ is a continuous map, modulo the relation that $(Z_0,f_0) \sim (Z_1,f_1)$ if the pair is bordant. Addition is defined by disjoint union
\[[Z_0,f_0] + [Z_1,f_1] := [Z_0 \sqcup Z_1,f_0 \sqcup f_1].\]
\end{definition}

\begin{definition}[Smooth bordism group of $X$] The \emph{$n$--th smooth bordism group} $\Omega_n^\infty(X;\Z/2)$ of a Fr\'echet manifold $X$ is group generated by equivalence classes $[Z,f]$ of pairs $(Z,f)$, where $Z$ is a closed $n$--dimensional manifold and $f:Z \to X$ is a smooth map, modulo the relation that $(Z_0,f_0) \sim (Z_1,f_1)$ if the pair is smoothly bordant. Addition in the group $\Omega_*^\infty(X;\Z/2)$ is defined by disjoint union as before.
\end{definition}

\begin{lemma} \label{lem:smooth_bordism_=_cts_bordism} The natural map $\Omega^\infty_*(X;\mathbb{Z}_2) \to \Omega_*(X;\mathbb{Z}_2)$ is an isomorphism.
\end{lemma}

\begin{proof} The argument uses smooth approximation and is identical to the case where $X$ is a finite dimensional smooth manifold, which can be found in \cite[Section I.9]{conner1964}.
\end{proof}

Given the above terminology, we can now prove the main result of this subsection, Proposition \ref{prop:bordism_vs_homology_with_sw_0}. It provides a class of submanifolds for which being null--bordant and being null--homologous are equivalent.

\begin{proposition} \label{prop:bordism_vs_homology_with_sw_0} Let $X$ be a metrizable Fr\'echet manifold, and let $f:Z \to X$ be a smooth map from a closed manifold $Z$ with Stieffel--Whitney class $w(Z) = 1 \in H^*(Z;\Z/2)$. Then $f_*[Z] = 0 \in H_*(X;\Z/2)$ if and only $[Z,f] = 0 \in \Omega_*^\infty(X;\mathbb{Z}_2)$.
\end{proposition}

\begin{proof} Proposition \ref{prop:bordism_vs_homology_with_sw_0} will follow immediately from the following results. First, by Lemma \ref{lem:smooth_bordism_=_cts_bordism}, it suffices to show $f_*[Z] = 0 \in H_*(X;\Z/2)$ if and only $[Z,f] = 0 \in \Omega_*(X;\mathbb{Z}_2)$. By Proposition \ref{prop:Frechet_is_CW}, we can replace $X$ with a CW complex. Lemma \ref{lem:bordism_vs_homology_with_sw_0} proves the result in this context.
\end{proof}

\begin{proposition}[{\cite[Theorem 14]{palais1966}}]
\label{prop:Frechet_is_CW} 
 A metrizable Fr\'echet manifold is homotopy equivalent to a CW complex.
\end{proposition}

\begin{lemma} \label{lem:bordism_vs_homology_with_sw_0} Let $X$ homotopy equivalent to a CW complex, and let $f:Z \to X$ be a continuous map from a closed manifold $Z$ with Stieffel--Whitney class $w(Z) = 1 \in H^*(Z;\Z/2)$. Then $f_*[Z] = 0 \in H_*(X;\Z/2)$ if and only $[Z,f] = 0 \in \Omega_*(X;\mathbb{Z}_2)$.
\end{lemma}

\begin{remark} Crucially, we make no finiteness assumptions on the CW structure.
\end{remark}

\begin{proof} ($\Rightarrow$) Suppose that $f_*[Z] = 0 \in H_2(Z;\Z/2)$. Pick a homotopy equivalence $\varphi:X \simeq X'$ with a CW complex $X'$. Such an equivalence induces an isomorphism of unoriented bordism groups $\Omega_*(X;\mathbb{Z}_2) \simeq \Omega_*(X';\mathbb{Z}_2)$, so it suffices to show that the pair $(Z,\varphi \circ f)$ is null--bordant, or equivalently to assume that $X$ is a CW complex to begin with.

So assume that $X$ is a CW complex. By Lemma \ref{lem:finite_subcomplex}, we can find a finite sub--complex $A \subset X$ such that $f(Z) \subset A$ and $f_*[Z] = 0 \in H_*(A;\Z/2)$. By Theorem 17.2 of \cite{conner1964}, $[Z,f] = 0 \in \Omega_*(A;\mathbb{Z}_2)$ if and only if the Stieffel--Whitney numbers $\op{sw}_{\alpha,I}[Z,f]$ are identically $0$. Recall that the Stieffel--Whitney number $\op{sw}_{\alpha,I}[Z,f]$ associated to $[Z,f]$, a cohomology class $\alpha \in H_k(A;\mathbb{Z}_2)$ and a partition $I = (i_1,\dots,i_k)$ of $\op{dim}(Z) - k$ is defined to be
\[
\op{sw}_{\alpha,I}[Z,f] = \langle w_{i_1}(Z)w_{i_2}(Z) \dots w_{i_k}(Z) f^*\alpha,[Z]\rangle \in \mathbb{Z}_2.
\]
Here $w_j(Z) \in H^j(Z;\mathbb{Z}_2)$ denotes the $j$--th Stieffel--Whitney class of $Z$. By assumption, $w(Z) = 1$ and so $w_j(Z) = 0$ for all $j \neq 0$. In particular, the only possible nonzero Stieffel--Whitney numbers have $I = (0)$. But we see that
\[
\op{sw}_{\alpha,(0)}[Z,f] = \langle f^*\alpha,[Z]\rangle = \langle \alpha,f_*[Z]\rangle = 0.
\]
Therefore, $\op{sw}_{\alpha,I}[Z,f] \equiv 0$ and $[Z,f]$ must be null--bordant.

($\Leftarrow$) This direction is completely obvious, since the map $\Omega_*(X) \to H_*(X;\Z/2)$ given by $[Z,f] \mapsto f_*[Z]$ is well defined.
\end{proof}

\begin{lemma} \label{lem:finite_subcomplex} Let $X$ be a CW complex, and let $f:Z \to X$ be a map from a closed manifold $Z$ with $f_*[Z] = 0 \in H_*(X;\Z/2)$. Then there exists a finite sub--complex $A \subset X$ with $f(Z) \subset A$ and $f_*[Z] = 0 \in H_*(A;\Z/2)$.
\end{lemma}

\begin{proof} A very convenient tool for this is the stratifold homology theory of $\cite{kreck2010}$, which we now review briefly. 

Given a space $M$, the $n$--th stratifold group $sH_n(M;\Z/2)$ with $\Z/2$--coefficients (see Proposition 4.4 in \cite{kreck2010}) is generated by equivalence classes of pairs $(S,g)$ of a compact, regular stratifold $S$ and a continuous map $g:S \to M$. Two pairs $(S_i,g_i)$ for $i \in \{0,1\}$ are equivalent if they are bordant by a $c$--stratifold, i.e. if there is a pair $(T,h)$ of a compact, regular $c$--stratifold and a continuous map $g:T \to M$ such that $(\partial T,h|_{\partial T}) = (S_0 \sqcup S_1,g_0 \sqcup g_1)$ (see Chapter 3 and Section 4.4 of \cite{kreck2010}). Given a map $\varphi:M \to N$ of spaces, the pushforward map $\varphi_*:sH(M;\mathbb{Z}_2) \to sH(M;\mathbb{Z}_2) $ on stratifold homology is given (on generators) by $[S,g] \mapsto [S,\varphi \circ g] = \varphi_*[\Sigma,g]$.

Stratifold homology satisfies the Eilenberg--Steenrod axioms (see Chapter 20 of \cite{kreck2010}), and thus if $M$ is a CW complex then there is a natural isomorphism $sH_*(M;\mathbb{Z}_2) \simeq H_*(M;\mathbb{Z}_2)$. If $M$ is a manifold of dimension $n$, the fundamental class $[M] \in sH_n(M;\mathbb{Z}_2)$ is given by the tautological equivalence class $[M] = [M,\op{Id}]$.

The proof of the lemma is simple with the above machinery in place. Since $f_*[Z] = 0$, the pair $(Z,f)$ must be null--bordant via some compact $c$--stratifold $(Y,g)$. Since $Y$ and its image $g(Y)$ are both compact, we can choose a sub--complex $A \subset X$ such that $g(T) \subset A \subset X$. Then the pair $(Z,f)$ are null--bordant by $(Y,g)$ in $A$ as well, so that $[Z,f] = 0 \in sH_*(A;\mathbb{Z}_2)$ and thus $f_*[Z] = 0 \in H_*(A;\mathbb{Z}_2)$ via the isomorphism $sH_*(A;\mathbb{Z}_2) \simeq H_*(A;\mathbb{Z}_2)$.\end{proof}

\subsection{Weinstein neighborhood theorem with boundary} \label{subsection:weinstein_neighborhood_w_bdry} In this section, we prove the analogue of the Weinstein neighborhood theorem for a Lagrangian $L$ with boundary, within a symplectic manifold $X$ with boundary. We could find no reference for this fact in the literature.

\begin{proposition}[Weinstein neighborhood theorem with boundary]
 \label{prop:weinstein_nbhd} 
 Let $(X,\omega)$ be a symplectic manifold with boundary $\partial X$ and let $L \subset X$ be a properly embedded, Lagrangain submanifold with boundary $\partial L \subset \partial X$ transverse to $T(\partial X)^\omega$.

Then there exists a neighborhood $U \subset T^*L$ of $L$ (as the zero section), a neighborhood $V \subset X$ of $L$ and a diffeomorphism $f:U \simeq V$ such that $\varphi^*(\omega|_V) = \omega_{\op{std}}|_U$. 
\end{proposition}

\begin{proof} The proof has two steps. First, we construct neighborhoods $U \subset T^*L$ and $V \subset X$ of $L$, and a diffeomorphism $\varphi:U \simeq V$ such that
\begin{equation} \label{equation:prop:weinstein_nbhd:1} 
\varphi|_L = \op{Id}, \qquad \varphi^*(\omega|_V)|_L = \omega_{\op{std}}|_L, \qquad T(\partial U)^{\omega_{\op{std}}} = T(\partial U)^{\varphi^*\omega}.
\end{equation}
Here $T(\partial U)^{\omega_{\op{std}}} \subset T(\partial U)$ is the symplectic perpendicular to $T(\partial U)$ with respect to $\omega_{\op{std}}$ (and similarly for $T(\partial U)^{\varphi^*\omega}$. Second, we apply Lemma \ref{lem:fiber_integration} and a Moser type argument to conclude the result.

(Step 1) Let $J$ be a compatible almost complex structure on $X$ and $g$ be the induced metric on $L$. Recall that the normal bundle $\nu_g L$ with respect to $g$ is a bundle over $L$ with Lagrangian fiber, and that $J:TL \to \nu_g L$ gives a natural isomorphism. Let $\Phi^g:T^*L \to TL$ denote the bundle isomorphism induced by the metric $g$ and let $\op{exp}^g$ denote the exponential map with respect to $g$. 

Since $L$ is compact, we can choose a tubular neighborhood $U'$ of $\nu L$ such that $\op{exp}^g:U \to X$ is a diffeomorphism onto its image $V$. We then let 
\[U := [J \circ \Phi^g]^{-1}(U') \subset T^*L\]
and also
\[
\phi^g:U \simeq V, \qquad (x,v) \mapsto \op{exp}_x^g(J \circ \Phi_g(v)).
\]
Note that $\phi^g|_L = \op{Id}$ and $[\phi^g]^*\omega|_L = \omega_{\op{std}}|_L$ by the same calculations as in \cite[Theorem 3.4.13]{mcduff2017}. We now must modify $U$, $V$, and $\phi^g$ to satisfy the last condition of (\ref{equation:prop:weinstein_nbhd:1}). 

To this end, we apply Lemma \ref{lem:trivial_foliation}. Taking $\kappa_0 = T(\partial U)^{\omega_{\op{std}}}$ and $\kappa_1 = T(\partial U)^{[\phi^g]^*\omega}$, we acquire a neighborhood $N \subset \partial(T^*L)$ of $\partial L$ and a family of embeddings $\psi:N \times I \to \partial(T^*L)$ with the following four properties:
\[\psi_t|_{\partial L} = \op{Id}, \qquad d(\psi_t)_u = \op{Id} \text{ for }u \in \partial L ,\qquad \psi_0 = \op{Id},\]
\[ [\psi_1]_*(T(\partial U)^{\omega_{\op{std}}}) = T(\partial U)^{[\phi^g]^*\omega}.\]
Note here that we are using the fact that $T(\partial U)^{\omega_{\op{std}}}|_L = T(\partial U)^{[\phi^g]^*\omega}|_L$ already by the construction of $\phi^g$. By shrinking $N$ and $U$, we can simply assume that $N = \partial U$. Let $\op{tc}:[0,1) \times \partial U \simeq T \subset U$ be tubular neighborhood coordinates near boundary. By choosing the tubular neighborhood coordinates $\op{tc}:[0,1) \times \partial U \simeq T$ appropriately, we can also assume that $\op{tc}([0,1) \times \partial L) = L \cap T$. We define a map $\Phi:U \to T^*L$ by
\[
\Phi(u) = \left\{\begin{array}{cc}
(s,\psi_{1-s}(v)) & \text{ if }u = (s,v) \in [0,1) \times \partial U \text{ via}\op{tc},\\
u & \text{ otherwise}.
\end{array}
\right.
\]
The map $\Phi$ has the following properties which are analogous to those of $\psi_s$:
\[
\Phi|_L = \op{Id}, \qquad d(\Phi)_u = \op{Id} \text{ for }u \in L, \qquad \Phi_*(T(\partial L)^{\omega_{\op{std}}}) = T(\partial L)^{[\phi^g]^*\omega}.
\]
Also note that $\Phi$ is smooth since $\psi_t$ is constant for $t$ near $0$ and $1$. We thus define $f$ as the composition $\varphi = \phi^g \circ \Phi$. It is immediate that $f$ has the properties in (\ref{equation:prop:weinstein_nbhd:1}).

(Step 2) We closely follows the Moser type argument of \cite[Lemma 3.2.1]{mcduff2017}. By shrinking $U$, we may assume that it is an open disk bundle. Let $\omega_t = (1-t)\omega_{\op{std}} + tf^*\omega$ and $\tau = \frac{d}{dt}(\omega_t) = f^*\omega - \omega_{\op{std}}$. Let $\kappa = T(\partial U)^{\omega_t}$ (by the previous work, it does not depend on $t$). Note that $\tau$ satisfies all of the assumptions of Lemma \ref{lem:fiber_integration}(\ref{equation:fiber_integration_tau}). We prove that $\kappa$ is invariant under the scaling map $\phi_t(x,u) = (x,tu)$ in Lemma \ref{lem:char_foliation_invt}. We can thus find a $\sigma$ satisfying the properties listed in (\ref{equation:fiber_integration_sigma}). 

Let $Z_t$ be the unique family of vector fields satisfying $\sigma = \iota(Z_t)\omega_t$. Due to the properties of $\sigma$, $Z_t$ satisfies the following properties for each $t$.
\[
Z_t|_L = 0,  \qquad Z_t|_{\partial U} \in T(\partial U) \text{ for all }t.
\]
The first property is immediate, while the latter is a consequence of the fact that
\[
\omega_t(Z_t,\cdot)|_{\kappa} = \sigma|_{\kappa} = 0 \]
implies
\[ Z_t \in (\kappa)^{\omega_t} = T(\partial U).
\]
These two properties imply that $Z_t$ generates a map $\Psi:U' \times [0,1] \to U$ for some smaller tubular neighborhood $U' \subset U$ with the property that $\Psi_t|_{L} = \op{Id}$ and $\Psi^*_t\omega_t = \omega_0$ (see \cite[\S 3.2]{mcduff2017}, as the reasoning is identical to the closed case). In particular, we get a map $\Psi_1:U' \to U$ with $\Psi_1|_L = \op{Id}$ and $\Psi_1^*f^*\omega$. By shrinking $U$, taking $\varphi = f \circ \Psi_1$ and taking $V = \varphi(U)$, we at last acquire the desired result. \end{proof}

The remainder of this section is devoted to proving the various lemmas that we used in the proof above.

\begin{lemma}[Fiber integration with boundary]
\label{lem:fiber_integration} 
Let $X$ be a compact manifold with boundary, $\pi:E \to X$ be a rank $k$ vector bundle with metric and $\pi:U \to X$ be the (open) disk bundle of $E$ with closure $\overline{U}$. Let $\kappa \subset T(\partial U)$ be a distribution on $\partial U$ such that $d\phi_{t}(\kappa_u) = \kappa_{\phi_t(u)}$ for all $u \in U$, where $\phi:U \times I \to U$ denote the family of smooth maps given by $\phi_t(x,u) := (x,tu)$. 

Finally, suppose that $\tau \in \Omega^{k+1}(\overline{U})$ is a $(k+1)$--form such that
\begin{equation} \label{equation:fiber_integration_sigma} d\tau = 0, \qquad \tau|_X = 0, \qquad (\iota^*_{\partial X}\tau)|_\kappa = 0.
\end{equation}
Then there exists a $k$--form $\sigma \in \Omega^k(\overline{U})$ with
\begin{equation} \label{equation:fiber_integration_tau} d\sigma = \tau, \qquad \sigma|_X = 0, \qquad (\iota^*_{\partial X}\sigma)|_\kappa = 0. \end{equation}
\end{lemma}

\begin{proof} We use integration over the fiber, as in \cite[p. 109]{mcduff2017}. Note that the maps $\phi_t:U \to \phi_t(U) \subset U$ are diffeomorphisms for each $t > 0$, $\phi_0 = \pi$, $\phi_1 = \op{Id}$ and $\phi_t|_X = \op{Id}$. Therefore we have
\[
\phi_0^*\tau = 0, \qquad \phi^*_1\tau = \tau.
\]
We may define a vector field $Z_t$ for all $t > 0$ and a $k$--form $\sigma_t$ for all $t \ge 0$ by
\[Z_t := (\frac{d}{dt}\phi_t) \circ \phi_t^{-1} \text{ for }t > 0, \qquad \sigma_t := \phi^*_t(\iota(Z_t)\tau) \text{ for }t \ge 0.\]
Although $Z_t$ is singular at $t = 0$, as in \cite{mcduff2017} one can verify in local coordinates that $\sigma_t$ is smooth at $t = 0$. Since $Z_t|_X = 0$, the $k$--form $\sigma_t$ satisfies $\sigma_t|_X = 0$. Furthermore, for any vector field $K \in \Gamma(\kappa)$ on $\partial X$ which is parallel to $\kappa$, we have $\iota(K)\sigma_t = \phi^*_t(\iota(Z_t)\iota(d\phi_t(K))\tau) = 0$ on the boundary, so that $\iota^*_{\partial X}(\sigma_t)|_{\kappa} = 0$. Finally, $\sigma_t$ satisfies the equation
\[
\begin{split}
\tau &= \phi^*_1\tau - \phi^*_0\tau = \int_0^1 \frac{d}{dt}(\phi^*_t\tau) dt = \int_0^1 \phi^*_t(\mathcal{L}_{X_t}\tau) dt \\
&= \int_0^1 d(\phi^*_t(\iota(X_t)\tau)) dt = \int_0^1 d\sigma_t dt = d(\int_0^1 \sigma_t dt).
\end{split}\]
Therefore, if we define $\sigma := \int_0^1 \sigma_t dt$, it is simple to verify the desired properties using the corresponding properties for $\sigma_t$.\end{proof}

\begin{lemma} \label{lem:trivial_foliation} Let $U$ be a manifold and $L \subset U$ be a closed submanifold. Let $\kappa_0,\kappa_1$ be rank $1$ orientable distributions in $TU$ such that $\kappa_i|_L \cap TL = \{0\}$ and $\kappa_0|_L = \kappa_1|_L$. 

Then there exists a neighborhood $U' \subset U$ of $L$ and a family of smooth embeddings $\psi:U' s\times I \to U$ with the following four properties:
\[
\psi_t|_{\partial L} = \op{Id}, \qquad d(\psi_t)_u = \op{Id} \text{ for }u \in L, \qquad \psi_0 = \op{Id}, \qquad [\psi_1]_*(\kappa_0) = \kappa_1.
\]
Furthermore, we can take $\psi_t$ to be $t$--independent for $t$ near $0$ and $1$.
\end{lemma}

\begin{proof} Since $\kappa_0$ and $\kappa_1$ are orientable, we can pick nonvanishing sections $Z_0$ and $Z_1$ We may assume that $Z_0 = Z_1$ along $L$. We let $Z_t$ denote the family of vector fields $Z_t := (1-t)Z_0 + tZ_1$. Since $Z_0 = Z_1$ along $L$, we can pick a neighborhood $N$ of $L$ such that $Z_t$ is nowhere vanishing for all $t$. We also select a submanifold $\Sigma \subset N$ with $\op{dim}(\Sigma) = \op{dim}(U) - 1$ and such that
\[
\Sigma \pitchfork Z_t \text{ for all }t \quad \text{and}\quad L \subset \Sigma.
\]
We can find such a $\Sigma$ by, say, picking a metric and using the exponential map on a neighborhood of $L$ in the sub--bundle $\nu L \cap \kappa_0^\perp$ of $TL$. By shrinking $\Sigma$ and scaling $Z_t$ to $\lambda Z_t$, $0 < \lambda < 1$, we can define a smooth family of embeddings
\[
\Psi:(-1,1)_s \times \Sigma \times [0,1]_t \to N, \qquad \Psi_t(s,x) = \op{exp}[Z_t]_s(x).
\]
Here $\op{exp}[Z_t]$ denotes the flow generated by $Z_t$. We let $\psi_t = \Psi_t \circ \Psi_0^{-1}$. To see the properties of (\ref{equation:prop:weinstein_nbhd:1}), note that $\Psi_t(0,l) = l$ for all $l \in L$ and $d(\Psi_t)_{0,l}(s,u) = sZ_t + u$. This implies the first two properties. The third is trivial, while the fourth is immediate from $[\Phi_t]_*(\partial_t) = Z_t$. We can make $\psi_t$ constant near $0$ and $1$ by simply reparametrizing with respect to $t$.
\end{proof}

\begin{lemma} \label{lem:char_foliation_invt} Let $L$ be a manifold with boundary and let $(T^*L,\omega)$ be the cotangent bundle with the standard symplectic form. Let $\kappa = T(\partial T^*L)^\omega$ denote the characteristic foliation of the boundary $\partial T^*L$ and let $\phi:T^*L \times (0,1] \to T^*L$ denote the family of maps $\phi_t(x,v) = (x,tv)$. Then $[\phi_t]_*(\kappa) = \kappa$.
\end{lemma}

\begin{proof} By passing to a chart, we may assume that $L \subset \R^+_{x_1} \times \R^{n-1}_x$ and $T^*L \subset \R^+_{x_1} \times \R^{n-1}_x \times \R^n_p$. Then $\kappa$ is simply given on $\partial T^*L \subset \{0\} \times \times \R^{n-1}_x \times \R^n_p$ by
\[
\kappa = \op{span}(\partial_{p_1}) = \op{span}(\partial_{p_2},\dots,\partial_{p_n},\partial_{x_1},\dots,\partial_{x_n})^\omega \subset T(\partial T^*L).
\]
Under the scaling map, we have $[\phi_t]_*(\partial_{p_1}) = t \cdot \partial_{p_1}$. This implies that $[\phi_t]_*(\kappa) = \kappa$. \end{proof}

\end{document}